\documentclass[12pt]{article}
\usepackage[T1]{fontenc}
\usepackage[a4paper,margin=2.55cm]{geometry}
\usepackage{amsmath,amssymb,amsthm,mathtools,mathrsfs}
\usepackage{microtype}
\usepackage[hidelinks]{hyperref}
\usepackage{authblk}
\numberwithin{equation}{section}
\newtheorem{theorem}{Theorem}[section]
\newtheorem{proposition}[theorem]{Proposition}
\newtheorem{lemma}[theorem]{Lemma}
\newtheorem{corollary}[theorem]{Corollary}
\theoremstyle{definition}
\newtheorem{definition}[theorem]{Definition}

\newcommand{\Wres}{\operatorname{Wres}}
\newcommand{\Wresb}{\widetilde{\operatorname{Wres}}}
\newcommand{\trE}{\operatorname{tr}_{E}}
\newcommand{\trg}{\operatorname{tr}_{g}}
\newcommand{\Id}{\operatorname{Id}}
\newcommand{\dd}{\,\mathrm d}
\newcommand{\ii}{\mathrm i}
\newcommand{\eps}{\varepsilon}
\newcommand{\iot}{\iota}
\newcommand{\vol}{\operatorname{vol}}
\newcommand{\Div}{\operatorname{div}}
\newcommand{\Hess}{\operatorname{Hess}}
\newcommand{\grad}{\operatorname{grad}}

\newcommand{\antic}[2]{\{#1,#2\}}

\title{Hamiltonian-Type Spectral Functionals for\protect\\
Time-Dependent Perturbed Hodge--de Rham Operators}
\author[1]{Sining Wei\thanks{%
E-mail: \texttt{weisn835@nenu.edu.cn}.}}

\author[2]{Yong Wang\thanks{%
Corresponding author.
E-mail: \texttt{wangy581@nenu.edu.cn}.}}

\affil[1]{%
School of Data Science and Artificial Intelligence\\
Dongbei University of Finance and Economics,\\
Dalian 116025, China}

\affil[2]{%
School of Mathematics and Statistics\\
Northeast Normal University,\\
Changchun 130024, China}

\date{}

\begin{document}
\maketitle

\begin{abstract}
Let $g_t$ be a smooth family of Riemannian metrics on an oriented
manifold of even dimension $n\geq4$, and let
$D_t=d+\delta_{g_t}+\Psi_t$ on the fixed exterior bundle, with
$\Psi_t$ self-adjoint. We evaluate a scalar-weighted residue
functional obtained from the operator expression in Hawkins'
Hamiltonian action using the coefficient derivative $\partial_tD_t$.
The velocity contribution depends only on $\partial_tg_t$.
A direct symbol calculation reduces the lapse correction to a
divergence and a Dirichlet term; for a nonconstant test function,
integration by parts retains a mixed gradient term.
For compact manifolds with a warped collar, we specify compatible
factorizations of the weighted KKW term and of the commutator
correction. Their boundary residues are computed explicitly.
The collar-derivative and perturbation contributions cancel in
the weighted KKW factorization, and the remaining boundary
terms involve normal derivatives of the scalar weights.
The closed formula is recovered when the boundary is empty.
\end{abstract}



\section{Introduction}
\label{sec:introduction}

The noncommutative residue relates the symbolic structure of an elliptic
operator to local Riemannian invariants. The Kastler--Kalau--Walze
theorem is a basic instance of this relation: a negative power of the
Dirac operator determines the Einstein--Hilbert integral
\cite{Kastler1995,KalauWalze1995}. For the Hodge--de Rham operator the
same construction is available on the exterior bundle, without a spin
structure. Replacing a single metric by a time-dependent family raises
a further question. Which geometric functional is obtained when the
Dirac operator and its time derivative enter the residue together?

A reason for considering this question comes from a spacelike
foliation. On a product $I\times M$, a Lorentzian metric with vanishing
shift can be written as $-N^2\dd t^2+g_t$, where each $g_t$ is positive
definite. The Dirac operators of the spatial slices are elliptic even
though the spacetime Dirac operator is not elliptic in the Riemannian
sense. Residues can therefore be calculated on $M$ at a fixed time.
The lapse $N$ and the variation of $g_t$ then carry information that is
absent from the ordinary, time-independent KKW formula. In this
description $n$ denotes the dimension of a spatial slice; the product
has dimension $n+1$.

Hawkins expressed the gravitational Lagrangian in terms of a family of
spatial Dirac operators, its evolution, and the lapse
\cite{Hawkins1997}. His formula contains a fourth-order velocity
insertion and a term of the form
$ND^2N^{-1}D^2N|D|^{-n-2}$. These terms suggest a residue calculation
for other Dirac-type operators. There is, however, a choice that has
to be made before such a calculation is performed. A family of bundles
and operators does not by itself specify a time derivative of the
operators. In Hawkins' construction the identification is obtained
from the spacetime Dirac equation. On the exterior bundle one can
instead use its metric-independent underlying vector bundle and
differentiate the coefficients of $d+\delta_{g_t}$. These two
derivatives need not agree. We use the latter derivative in this paper,
and use the term \emph{Hamiltonian-type} to distinguish the resulting
functional from a reconstruction of the ADM action.

Several related constructions provide the analytic and geometric
background. The original residue formulas of Kastler and of Kalau and
Walze identify the scalar-curvature term for spin Dirac operators
\cite{Kastler1995,KalauWalze1995}; the heat-coefficient interpretation
and generalized Lichnerowicz formulas explain how a zero-order
perturbation changes that term \cite{Ackermann1996,AckermannTolksdorf1996}.
Hawkins' formula gives the particular combination of operator products
that we investigate \cite{Hawkins1997}. Van den Dungen develops the
relation between families of spectral triples and foliations of
space(time), including the operator-theoretic role of a lapse
\cite{Dungen2018}. These results motivate the use of an elliptic
operator on each slice, but do not identify arbitrary choices of
temporal differentiation. The dependence of Dirac operators on the
identification of bundles under metric variation is also central in
the work of Bourguignon and Gauduchon \cite{BourguignonGauduchon1992}.
For the Hodge operator, D\k{a}browski, Sitarz, and Zalecki calculate
spectral metric and Einstein functionals \cite{DSZ2025}, and Bochniak,
D\k{a}browski, Sitarz, and Zalecki treat geometric spectral functionals
for perturbed operators and for torsion \cite{BDSZ2026}. Their
scalar-curvature formulas supply the static curvature terms used
here; we give the corresponding traced-potential calculation in
our conventions. Spectral forms for perturbed Hodge operators are
also studied in \cite{WangWangLiu2025}.

Our setting is a fixed closed oriented manifold $M$, a smooth family
$g_t$, and a smooth positive function $N$ on $I\times M$. We work on
$E=\Lambda^*T^*M\otimes\mathbb C$ with
$D_t=d+\delta_{g_t}+\Psi_t$, where $\Psi_t$ is pointwise self-adjoint
for the metric induced by $g_t$. The derivative $\dot D_t$ is taken in
the fixed bundle $E$. For a smooth real test function $f$, the functional $\mathscr L_t(f)$ is defined in
Definition~\ref{def:functional}. It uses the coefficients in
\cite[(8)]{Hawkins1997}, with the ordinary fiber trace and the residue
normalization specified below. The principal results evaluate this
functional, rather than assume that it has a prescribed classical
interpretation.

Theorem~\ref{thm:general-main} expresses this functional in terms of
the metric velocity, scalar curvature, a quadratic contraction of
the perturbation, and spatial derivatives of the lapse and test
function. The velocity contribution depends only on the metric
variation. Neither temporal nor spatial derivatives of the
perturbation remain in the final formula. For a nonconstant test
function, the divergence in the local lapse term contributes a
mixed gradient term after integration by parts.

We compute the lapse correction in the coordinate exterior frame,
retaining the first-order coefficient of the perturbed square even
at a normal-coordinate center. Its contributions cancel after
spherical integration. When the test function is constant, the
trace property and Hawkins' commutator rearrangement
\cite[Section~7, preceding (8)]{Hawkins1997} give an independent
verification of the integrated correction.

For compact spatial slices with boundary, we use the
noncommutative residue of \cite{FGLS1996} and the boundary symbol
method of \cite{Wang2007,WangWang2025Torsion}. With the
factorizations specified in Section~\ref{sec:boundary}, the
collar-derivative and perturbation terms cancel in the weighted
KKW contribution. The remaining singular Green residues involve
normal derivatives of the scalar weights. Combining these
residues with the boundary term from integration by parts gives
Theorem~\ref{thm:boundary-main}.

The paper is organized as follows.
Section~\ref{sec:preliminaries} gives the geometric and analytic
preliminaries.
Section~\ref{sec:residue-calculations} computes the three interior
residue contributions, derives the closed Hamiltonian-type formula,
and treats metric-connection perturbations.
Section~\ref{sec:boundary} computes the boundary residues and
proves the corresponding formula for manifolds with boundary.

\section{Preliminaries}
\label{sec:preliminaries}

We give the conventions for the exterior bundle and for differentiating
a family of operators. We then specify the residue functional and the
heat-coefficient normalization needed for its evaluation. Unless a time
derivative is displayed, every differential operator and every
geometric quantity is evaluated on a single slice.

\subsection{The exterior bundle and a family of Hodge operators}

Let $M$ be a closed oriented smooth manifold of dimension $n=2m\geq4$,
and let $I$ be an interval. A smooth family of Riemannian metrics on
$M$ is denoted by $g_t$. The complex vector bundle
$E=\Lambda^*T^*M\otimes\mathbb C$ has rank $2^n$ and does not depend
on $t$. Its Hermitian metric and its $L^2$ inner product do depend on
$g_t$. The notation $\trE$ always means the ordinary trace on the
whole fiber, so that $\trE(\Id_E)=2^n$.

Let $\alpha,\beta$ be one-forms, let $X,Y$ be vector fields, and let
$\omega$ be a $p$-form. Exterior multiplication is
$\eps(\alpha)\omega=\alpha\wedge\omega$.
For $p\geq1$, interior multiplication is defined by
$(\iot(X)\omega)(X_1,\ldots,X_{p-1})=
\omega(X,X_1,\ldots,X_{p-1})$, and it is zero on functions.
These operations satisfy
\begin{align}
 \eps(\alpha)\eps(\beta)+\eps(\beta)\eps(\alpha)&=0,\notag\\
 \iot(X)\iot(Y)+\iot(Y)\iot(X)&=0,\notag\\
 \eps(\alpha)\iot(X)+\iot(X)\eps(\alpha)
 &=\alpha(X)\Id_E.\label{eq:CAR}
\end{align}
Both operations are defined before a metric is chosen. For a metric
$g$, the vector $\alpha^{\sharp_g}$ is determined by
$g(\alpha^{\sharp_g},X)=\alpha(X)$. We use Clifford multiplication
by covectors in the form
\begin{equation*}
 c_g(\alpha)=\eps(\alpha)-\iot(\alpha^{\sharp_g}),
 \qquad
 \antic{c_g(\alpha)}{c_g(\beta)}
 =-2g^{-1}(\alpha,\beta)\Id_E.
\end{equation*}
Here $\{A,B\}=AB+BA$. The two terms in $c_g$ are adjoints of one
another, and consequently $c_g(\alpha)^*=-c_g(\alpha)$ for real
$\alpha$.

In local coordinates $x=(x^1,\ldots,x^n)$ we write
$\partial_i=\partial/\partial x^i$ and
$g_{ij}=g(\partial_i,\partial_j)$; $(g^{ij})$ is the inverse matrix.
Indices range from $1$ to $n$, and repeated indices are summed.
The induced Levi--Civita connection on $E$ is denoted by
$\nabla^{E,g}$. We use its standard exterior representation
\cite{BGV1992}; on the coordinate exterior frame it has the form
\begin{equation}
\label{eq:exterior-connection}
 \nabla^{E,g}_{\partial_i}=\partial_i+\omega_i,
 \qquad
 \omega_i=-\Gamma(g)^a{}_{ib}\eps(dx^b)\iot(\partial_a).
\end{equation}
Indeed, the dual connection gives
$\nabla_{\partial_i}dx^a=-\Gamma(g)^a{}_{ib}dx^b$.
The operator $\eps(dx^b)\iot(\partial_a)$ replaces each occurrence
of $dx^a$ in a basis form by $dx^b$. Its sign from contraction is
cancelled by moving $dx^b$ back to that position. The Leibniz rule on
exterior products therefore gives \eqref{eq:exterior-connection}.

The codifferential is the formal adjoint $\delta_g$ of $d$ for this
metric. The torsion-free identities are
\begin{equation}
\label{eq:Hodge-coordinate}
 d=\eps(dx^i)\partial_i,\qquad
 \delta_g=-g^{ij}\iot(\partial_i)\nabla^{E,g}_{\partial_j},
 \qquad D_g^0=d+\delta_g.
\end{equation}
For an increasing multi-index $J=(j_1,\ldots,j_p)$, write
$dx^J=dx^{j_1}\wedge\cdots\wedge dx^{j_p}$. If
$\alpha=\sum_J a_Jdx^J$, then
\begin{align*}
 d\alpha
 &=\sum_J\bigl(da_J\wedge dx^J+a_Jd(dx^J)\bigr)\\
 &=\sum_{i,J}(\partial_i a_J)dx^i\wedge dx^J
 =\eps(dx^i)\partial_i\alpha,
\end{align*}
because $d(dx^J)=0$. By the symmetry of the Christoffel symbols,
\[
 \sum_j\eps(dx^j)\omega_j
 =-\Gamma(g)^a{}_{jb}\eps(dx^j)\eps(dx^b)\iot(\partial_a)=0.
\]
Consequently, \eqref{eq:Hodge-coordinate} gives
\[
 D_g^0
 =\bigl(\eps(dx^j)-g^{ij}\iot(\partial_i)\bigr)
                      \nabla^{E,g}_{\partial_j}
 =c_g(dx^j)\nabla^{E,g}_{\partial_j}.
\]
With $\sigma(\partial_j)=\ii\xi_j$ and
$\xi=\xi_jdx^j$, we obtain
\[
 \sigma_1(D_g^0)=\ii c_g(\xi),
 \qquad
 \sigma_2((D_g^0)^2)
 =(\ii c_g(\xi))^2=|\xi|_g^2\Id_E.
\]

We consider a smooth endomorphism $\Psi_t$ of $E$ that is
self-adjoint for $g_t$, and set
\begin{equation}
\label{eq:Dt}
 D_t=D_{g_t}^0+\Psi_t.
\end{equation}
At each fixed time this is a formally self-adjoint elliptic
first-order operator. Its self-adjoint closure on
$L^2(M,E;g_t)$ has compact resolvent. A zero-order perturbation is
bounded on this Hilbert space, so it does not change the domain of
the closure. We do not require $\Psi_t$ to be odd for the exterior
parity decomposition. When it is odd, $D_t$ is an odd operator as well.

At a fixed time, write $g=g_t$ and $\Psi=\Psi_t$, and define
\begin{equation}
\label{eq:vPsi}
 v_\Psi=2^{-n}\trE\left(
 \Psi^2+\frac14g_{ij}
       \antic{c_g(dx^i)}{\Psi}\antic{c_g(dx^j)}{\Psi}
 \right).
\end{equation}
All products are endomorphism compositions. For a
$g$-orthonormal coframe $e^1,\ldots,e^n$ at a point,
\[
 g_{ij}\antic{c_g(dx^i)}{\Psi}\antic{c_g(dx^j)}{\Psi}
 =\sum_a\antic{c_g(e^a)}{\Psi}^{\,2},
\]
so the contraction in \eqref{eq:vPsi} is independent of coordinates.
The adjoint relations are
\[
 \Psi^*=\Psi,\qquad c_g(dx^i)^*=-c_g(dx^i),\qquad
 \antic{c_g(dx^i)}{\Psi}^{*}=-\antic{c_g(dx^i)}{\Psi}.
\]
Hence cyclicity gives
\begin{align*}
 \overline{\trE\!\left(
    \antic{c_g(dx^i)}{\Psi}\antic{c_g(dx^j)}{\Psi}\right)}
 &=\trE\!\left(
    \antic{c_g(dx^j)}{\Psi}^{*}
    \antic{c_g(dx^i)}{\Psi}^{*}\right)\\
 &=\trE\!\left(
    \antic{c_g(dx^j)}{\Psi}\antic{c_g(dx^i)}{\Psi}\right)\\
 &=\trE\!\left(
    \antic{c_g(dx^i)}{\Psi}\antic{c_g(dx^j)}{\Psi}\right).
\end{align*}
Together with $\overline{\trE\Psi^2}=\trE\Psi^2$ and
$g_{ij}\in\mathbb R$, this proves $\overline{v_\Psi}=v_\Psi$.
No sign condition is imposed on $v_\Psi$.

For the Levi--Civita connection $\nabla^g$, we use
\[
 R^g(X,Y)Z
 =\nabla_X^g\nabla_Y^g Z-\nabla_Y^g\nabla_X^g Z
                   -\nabla_{[X,Y]}^g Z.
\]
The Ricci tensor and scalar curvature are its contractions:
\[
 \operatorname{Ric}_g(Y,Z)
 =g^{ij}g(R^g(\partial_i,Y)Z,\partial_j),
 \qquad
 R_g=g^{ij}\operatorname{Ric}_g(\partial_i,\partial_j).
\]
Equivalently, for a $g$-orthonormal basis $e_1,\ldots,e_n$
at a point,
\[
 \operatorname{Ric}_g(Y,Z)=\sum_a g(R^g(e_a,Y)Z,e_a),
 \qquad R_g=\sum_a\operatorname{Ric}_g(e_a,e_a).
\]
For the Hodge operator the Weitzenb\"ock formula is
\begin{equation}
\label{eq:Weitzenbock}
 (D_g^0)^2=(\nabla^{E,g})^*\nabla^{E,g}+\mathcal R_g,
 \qquad \trE\mathcal R_g=2^{n-2}R_g.
\end{equation}
The trace on $p$-forms is
$\binom{n-2}{p-1}R_g$, with a binomial coefficient equal to zero
outside its range. Summing over $p$ gives the second identity.
These are the usual Hodge heat-kernel conventions
\cite{BGV1992,Gilkey1995,DSZ2025}.

\subsection{Temporal differentiation and the functional}

The differentiable structure of $M$ and the bundle $E$ are fixed.
For a time-dependent section $s_t$, its derivative is taken in this
fixed bundle. We define
$(\dot D_t)s=\partial_t(D_ts)$ for a section $s$ independent of $t$.
Thus, if in a fixed local frame
$D_t=A^j(t,x)\partial_j+B(t,x)$, then
$\dot D_t=(\partial_tA^j)\partial_j+\partial_tB$.
Transition functions between such frames do not depend on $t$;
this definition is consequently global. It does not assert that
$\dot D_t$ is self-adjoint in $L^2(M,E;g_t)$.

Write $\dot g_t=\partial_tg_t$. Differentiating
$g_t^{ia}g_{t,aj}=\delta^i_j$ yields
\[
 (\partial_tg_t^{ia})g_{t,aj}
       +g_t^{ia}(\partial_tg_{t,aj})=0.
\]
Multiplying by $g_t^{jb}$ and relabelling the free indices gives
\begin{equation}
\label{eq:inverse-velocity}
 \partial_tg_t^{ij}
 =-g_t^{ia}(\partial_tg_{t,ab})g_t^{bj}.
\end{equation}
At fixed time we abbreviate $g_t,\dot g_t,D_t,\Psi_t$ by
$g,\dot g,D,\Psi$. With $\dot g_{ij}=\partial_tg_{t,ij}$,
\[
 \trg\dot g=g^{ij}\dot g_{ij},
 \qquad
 |\dot g|_g^2=g^{ia}g^{jb}\dot g_{ij}\dot g_{ab}.
\]
All differentials of $N$ without a dot are spatial:
$\dd N=(\partial_iN)dx^i$.

For the self-adjoint operator $D$, set $|D|=(D^2)^{1/2}$.
If $D\phi_\lambda=\lambda\phi_\lambda$, spectral calculus gives
\[
 |D|\phi_\lambda=|\lambda|\phi_\lambda,
 \qquad
 |D|^{-a}\phi_\lambda=
 \begin{cases}
 |\lambda|^{-a}\phi_\lambda,&\lambda\ne0,\\
 0,&\lambda=0,
 \end{cases}
 \quad a>0.
\]
Thus $|D|^{-a}$ is a classical pseudodifferential operator of
order $-a$ \cite{Seeley1967}. For the even powers used below,
it may equivalently be replaced by a parametrix of the
corresponding power of $D^2$. Spectral calculus also gives
$D^2|D|^{-a}=|D|^{-a}D^2$; the parametrix identity holds modulo
smoothing operators. All calculations are at fixed $t$ and
do not differentiate the inverse powers.

We apply the operator expression of \cite[(8)]{Hawkins1997}
to the Hodge family \eqref{eq:Dt}, using the fixed-bundle
coefficient derivative defined above.

\begin{definition}[Hamiltonian-type residue functional]
\label{def:functional}
For the family \eqref{eq:Dt}, a smooth positive function $N$ on
$I\times M$, and a smooth real function $f$ on $I\times M$, define
\begin{align}
 \mathscr L_t(f)={}&
 \Wres\!\left[
 fN^{-1}|D|^{-n-4}
 \left(\frac{n(n+2)}8(D\dot D)^2
       +\frac{n(n-4)}8D^2\dot D^2\right)\right]\notag\\
 &+\left(\frac6{n-2}+\frac{n(n-2)}{16}\right)
       \Wres(fN|D|^{2-n})\notag\\
 &-\frac{n(n-2)}{16}
       \Wres(fND^2N^{-1}D^2N|D|^{-n-2}).
 \label{eq:functional}
\end{align}
The products are compositions in the displayed order and
$(D\dot D)^2=D\dot D D\dot D$. On a compact interval
$[t_0,t_1]\subset I$ the action is
$\mathscr A(f)=\int_{t_0}^{t_1}\mathscr L_t(f)\dd t$.
The test function $f$ is real and may change sign. At each time
$\mathscr L_t(af_1+bf_2)=a\mathscr L_t(f_1)+b\mathscr L_t(f_2)$
for $a,b\in\mathbb R$, and $f=1$ gives the unweighted functional.
\end{definition}

We note that, for a general family spectral triples as in \cite{Dungen2018}, we may define a similar Hamiltonian
gravity functional. When $D$ is the Dirac operator and $f=1$, we get the Hamiltonian
gravity in \cite{Hawkins1997}.

Throughout, $n=2m\geq4$ and $\dot D_t=\partial_tD_t$ in the
fixed exterior bundle. In particular,
\[
 |D|^{2-n}=(D^2)^{1-m},\qquad
 |D|^{-n-2}=(D^2)^{-m-1},\qquad
 |D|^{-n-4}=(D^2)^{-m-2}.
\]
More generally, \eqref{eq:functional} defines a linear functional
for a smoothly identified family of spectral triples whenever
$\dot D_t$ and a residue trace on the displayed products are defined.

\subsection{Symbols, residues, and heat coefficients}

In a local trivialization, the full symbol of a classical operator
$A$ of integer order $a$ has homogeneous components
$\sigma(A)\sim\sum_{j\geq0}\sigma_{a-j}(A)$ for $\xi\ne0$.
For the left symbol convention the standard composition rule
\cite{Gilkey1995} is
\begin{equation}
\label{eq:composition}
 \sigma(AB)\sim
 \sum_\alpha\frac{1}{\alpha!}
 (\partial_\xi^\alpha\sigma(A))
 (D_x^\alpha\sigma(B)),
 \qquad D_{x_j}=\frac1\ii\partial_{x_j}.
\end{equation}
Here $\alpha=(\alpha_1,\ldots,\alpha_n)$ is a multi-index of
nonnegative integers, $|\alpha|=\sum_j\alpha_j$, and
$\alpha!=\prod_j\alpha_j!$. Differentiation in $\xi$ lowers the
homogeneity by $|\alpha|$; differentiation in $x$ does not.

For a smooth scalar multiplier $f$ and a classical operator $T$,
the composition rule gives
\[
 \sigma(fT)
 \sim\sum_\alpha\frac1{\alpha!}
          (\partial_\xi^\alpha f)D_x^\alpha\sigma(T)
 =f\sigma(T),
\]
since $\partial_\xi^\alpha f=0$ for $|\alpha|>0$.

We use the noncommutative residue in the normalization
\begin{equation}
\label{eq:Wres}
 \Wres(A)=\int_M\int_{|\xi|_g=1}
       \trE\sigma_{-n}(A)(x,\xi)
       \dd S_g(\xi)\dd\vol_g(x).
\end{equation}
The expression defines the invariant residue density; normal
coordinates and a local bundle frame can be used to compute it.
In particular, the result does not depend on the auxiliary metric
used to describe the cosphere. We write $\nu_{n-1}=\vol(S^{n-1})$.
The residue is a trace and vanishes on smoothing operators
\cite{Wodzicki1987}. It also vanishes on any classical operator
of order strictly smaller than $-n$.

\begin{lemma}[Spherical moments]
\label{lem:spherical-moments}
Let $\dd S$ be Euclidean surface measure on $S^{n-1}$ and
$\nu_{n-1}=\vol(S^{n-1})$. Then
\begin{align}
 \int_{S^{n-1}}\xi_i\xi_j\dd S
 &=\frac{\nu_{n-1}}n\delta_{ij},\notag\\
 \int_{S^{n-1}}\xi_i\xi_j\xi_k\xi_l\dd S
 &=\frac{\nu_{n-1}}{n(n+2)}
 (\delta_{ij}\delta_{kl}+\delta_{ik}\delta_{jl}
                         +\delta_{il}\delta_{jk}).
 \label{eq:sphere-moments}
\end{align}
For a symmetric real matrix $H=(H_{ij})$, put
$\operatorname{tr}H=\sum_iH_{ii}$ and
$\operatorname{tr}(H^2)=\sum_{i,j}H_{ij}H_{ji}$. Then
\begin{equation}
\label{eq:sphere-h}
 \int_{S^{n-1}}\left(\sum_{i,j}H_{ij}\xi_i\xi_j\right)^2\dd S
 =\frac{\nu_{n-1}}{n(n+2)}
       \bigl((\operatorname{tr}H)^2+2\operatorname{tr}(H^2)\bigr).
\end{equation}
No definiteness assumption on $H$ is required.
\end{lemma}

\begin{proof}
Invariance under orthogonal transformations gives
$\int_{S^{n-1}}\xi_i\xi_j\dd S=a_2\delta_{ij}$ and
\[
 \int_{S^{n-1}}\xi_i\xi_j\xi_k\xi_l\dd S
 =a_4(\delta_{ij}\delta_{kl}
                 +\delta_{ik}\delta_{jl}+\delta_{il}\delta_{jk}).
\]
Since $\sum_i\xi_i^2=1$ on $S^{n-1}$, contraction yields
\begin{align*}
 na_2
 &=\int_{S^{n-1}}\sum_i\xi_i^2\dd S=\nu_{n-1},\\
 (n^2+2n)a_4
 &=\int_{S^{n-1}}\left(\sum_i\xi_i^2\right)^2\dd S
 =\nu_{n-1}.
\end{align*}
These identities prove \eqref{eq:sphere-moments}. Expanding
the quadratic form and applying its fourth-moment identity gives
\begin{align*}
 \int_{S^{n-1}}\left(\sum_{i,j}H_{ij}\xi_i\xi_j\right)^2\dd S
 &=\sum_{i,j,k,l}H_{ij}H_{kl}
          \int_{S^{n-1}}\xi_i\xi_j\xi_k\xi_l\dd S\\
 &=\frac{\nu_{n-1}}{n(n+2)}
   \left[\sum_{i,k}H_{ii}H_{kk}
        +\sum_{i,j}H_{ij}^2+\sum_{i,j}H_{ij}H_{ji}\right]\\
 &=\frac{\nu_{n-1}}{n(n+2)}
       \bigl((\operatorname{tr}H)^2+2\operatorname{tr}(H^2)\bigr),
\end{align*}
where the last equality uses $H_{ij}=H_{ji}$.
\end{proof}

At a $g$-normal-coordinate center the lemma applies to the
fixed matrix $H_{ij}=\dot g_{ij}$. In this case
$\operatorname{tr}H=\trg\dot g$ and
$\operatorname{tr}(H^2)=|\dot g|_g^2$.

The weighted Kastler--Kalau--Walze formula
\cite{Kastler1995,KalauWalze1995} takes the following form in
the normalization \eqref{eq:Wres}.

\begin{lemma}[Weighted residue and the second heat coefficient]
\label{lem:heat-residue}
Let $A$ be a nonnegative self-adjoint Laplace-type operator on $E$,
with scalar principal symbol $|\xi|_g^2\Id_E$. Write its Bochner
decomposition as $A=\nabla^*\nabla+V$. For a smooth scalar function $f$,
\begin{equation}
\label{eq:heat-residue}
 \Wres(fA^{-(n-2)/2})
 =\frac{n-2}{2}\nu_{n-1}
 \int_M f\,\trE\left(\frac16R_g\Id_E-V\right)\dd\vol_g.
\end{equation}
\end{lemma}

\begin{proof}
The residue--heat coefficient identity \cite{Ackermann1996},
with the standard
second coefficient for $A=\nabla^*\nabla+V$
\cite{BGV1992,Gilkey1995}, gives the integral in
\eqref{eq:heat-residue} with coefficient
\[
 \frac{2(2\pi)^n(4\pi)^{-n/2}}{\Gamma((n-2)/2)}
 =\frac{2\pi^{n/2}}{\Gamma((n-2)/2)}
 =\frac{n-2}{2}\nu_{n-1}.
\]
Here $\nu_{n-1}=2\pi^{n/2}/\Gamma(n/2)$. The scalar
multiplier $f$ multiplies the diagonal heat coefficient,
so no derivative of $f$ occurs. The kernel projection is smoothing
and does not change the residue.
\end{proof}

\section[The interior residues and the Hamiltonian-type formula]
{The interior residues and the Hamiltonian-type formula}
\label{sec:residue-calculations}

We evaluate the velocity, curvature, and lapse terms in
\eqref{eq:functional} and combine them to obtain the closed
Hamiltonian-type formula. The velocity terms have total order $-n$
and are determined by principal symbols. The weighted curvature
term uses the second heat coefficient. For the lapse correction,
we give a direct third-order symbol calculation and a verification
using Hawkins' commutator rearrangement. The final subsection
derives the total formula and identifies the zero-order
perturbation induced by a change of metric connection.

\subsection{The velocity terms}

We first record the full variation of the operator. This
identifies the terms that will subsequently disappear by
homogeneity and avoids imposing a special spatial constancy
assumption on the family $g_t$.

\begin{lemma}[Variation in a fixed coordinate exterior frame]
\label{lem:full-variation}
Let $\dot g=\partial_tg$.
The derivative of the Levi--Civita connection has components
\begin{equation}
\label{eq:Gamma-variation}
 \partial_t\Gamma(g)^a{}_{jb}
 =\frac12g^{al}
       \left(\nabla_j\dot g_{bl}+\nabla_b\dot g_{jl}
                                      -\nabla_l\dot g_{jb}\right).
\end{equation}
Here $\nabla$ denotes the Levi--Civita covariant derivative
for the metric at the given time. The full operator derivative is
\begin{equation}
\label{eq:full-D-variation}
 \dot D=g^{ia}g^{jb}\dot g_{ab}\iot(\partial_i)\nabla^{E,g}_{\partial_j}
 +g^{ij}\partial_t\Gamma(g)^a{}_{jb}
           \iot(\partial_i)\eps(dx^b)\iot(\partial_a)
 +\dot\Psi.
\end{equation}
\end{lemma}

\begin{proof}
For vector fields independent of time, set
$C(X,Y)=\partial_t(\nabla^g_XY)$ during this proof.
Differentiating the torsion-free identity gives
$C(X,Y)=C(Y,X)$. Metric compatibility gives
\[
 X(g(Y,Z))=g(\nabla_XY,Z)+g(Y,\nabla_XZ).
\]
Differentiation in $t$, followed by the definition of the
covariant derivative of $\dot g$, yields
\[
 (\nabla_X\dot g)(Y,Z)=g(C(X,Y),Z)+g(C(X,Z),Y).
\]
The corresponding identities with $X,Y$ interchanged
and with $X,Z$ interchanged imply
\begin{align*}
 2g(C(X,Y),Z)
 ={}&(\nabla_X\dot g)(Y,Z)+(\nabla_Y\dot g)(X,Z)
                         -(\nabla_Z\dot g)(X,Y).
\end{align*}
Indeed, the two terms paired with $Y$ cancel by the
symmetry of $C$, and the same is true of the two terms
paired with $X$. Taking $X=\partial_j$, $Y=\partial_b$,
and $Z=\partial_l$, and then contracting with $g^{al}/2$,
proves \eqref{eq:Gamma-variation}.

The exterior derivative is independent of $t$. Using
\eqref{eq:Hodge-coordinate} and the fact that
$\iot(\partial_i)$ is independent of $t$, we obtain
\begin{align*}
 \partial_t\delta_g
 &=-\partial_tg^{ij}\iot(\partial_i)
                         \nabla^{E,g}_{\partial_j}
       -g^{ij}\iot(\partial_i)\partial_t\omega_j\\
 &=g^{ia}g^{jb}\dot g_{ab}\iot(\partial_i)\nabla^{E,g}_{\partial_j}
       +g^{ij}\partial_t\Gamma(g)^a{}_{jb}
                    \iot(\partial_i)\eps(dx^b)\iot(\partial_a).
\end{align*}
The second equality uses \eqref{eq:exterior-connection}
and \eqref{eq:inverse-velocity}. Adding $\dot\Psi$
proves \eqref{eq:full-D-variation}.
\end{proof}

Thus the spatial derivative of $\dot g$ occurs through
$\partial_t\Gamma$, but only as a zero-order coefficient
of $\dot D$. The endomorphism $\dot\Psi$ is another
zero-order coefficient. The first-order coefficient in
\eqref{eq:full-D-variation} depends on $\dot g$ without its
spatial derivatives. This is the precise separation
used in the following principal-symbol calculation.

Fix $t$ and a point $x_0$. Coordinates are held fixed when taking
$\partial_t$, and may be chosen to be normal for $g_t$ at this
particular time and point. Expanding
\eqref{eq:full-D-variation} with
$\nabla^{E,g}_{\partial_j}=\partial_j+\omega_j$ gives
\begin{align*}
 \dot D
 ={}&-(\partial_tg^{ij})\iot(\partial_i)\partial_j+\partial_t\bigl(g^{ij}\Gamma(g)^a{}_{jb}\bigr)
          \iot(\partial_i)\eps(dx^b)\iot(\partial_a)
   +\dot\Psi .
\end{align*}
In particular, the zero-order symbol at the normal-coordinate
center is
\[
 \sigma_0(\dot D)(x_0)
 =g^{ij}(x_0)\partial_t\Gamma(g)^a{}_{jb}(x_0)
          \iot(\partial_i)\eps(dx^b)\iot(\partial_a)
       +\dot\Psi(x_0).
\]
Here $\Gamma(g)(x_0)=0$, whereas
$\partial_t\Gamma(g)(x_0)$ need not vanish.
For a covector $\xi=\xi_jdx^j$, define the vector
$w=(\partial_tg^{ij})\xi_j\partial_i$. The principal symbols are
\begin{equation}
\label{eq:velocity-symbol}
 \sigma_1(D)=\ii c_g(\xi),\qquad
 \sigma_1(\dot D)=-\ii\iot(w),\qquad
 \xi(w)=-\dot g(\xi^{\sharp_g},\xi^{\sharp_g}).
\end{equation}
The notation $w$ is used only in the pointwise calculation below.

\begin{lemma}[Two principal-symbol contractions]
\label{lem:velocity-traces}
At $(x_0,\xi)$ one has
\begin{align}
 \trE\bigl(\sigma_1(D)\sigma_1(\dot D)
            \sigma_1(D)\sigma_1(\dot D)\bigr)
 &=2^{n-1}\bigl[\dot g(\xi^{\sharp_g},\xi^{\sharp_g})\bigr]^2,
 \label{eq:velocity-four}\\
 \sigma_1(\dot D)^2&=0.\notag
\end{align}
\end{lemma}

\begin{proof}
Interior multiplication satisfies $\iot(w)^2=0$ by
\eqref{eq:CAR}. Moreover,
$\iot(w)c_g(\xi)+c_g(\xi)\iot(w)=\xi(w)\Id_E$.
It follows that
\begin{align*}
 (c_g(\xi)\iot(w))^2
 &=c_g(\xi)\bigl(\xi(w)\Id_E-c_g(\xi)\iot(w)\bigr)\iot(w)\\
 &=\xi(w)c_g(\xi)\iot(w).
\end{align*}
The product $\iot(\xi^{\sharp_g})\iot(w)$ lowers form degree by
two and has trace zero. Taking the trace in the third identity of
\eqref{eq:CAR}, and using cyclicity, gives
$\trE(\eps(\xi)\iot(w))=2^{n-1}\xi(w)$.
Consequently,
\[
 \trE(c_g(\xi)\iot(w)c_g(\xi)\iot(w))
 =2^{n-1}\xi(w)^2.
\]
The factors $\ii$ and $-\ii$ in \eqref{eq:velocity-symbol}
have product one. Substitution of the last identity of
\eqref{eq:velocity-symbol} proves \eqref{eq:velocity-four}.
\end{proof}

\begin{proposition}[Velocity residues]
\label{prop:velocity}
For the coefficient derivative on the fixed exterior bundle,
\begin{align}
 \Wres\bigl(fN^{-1}|D|^{-n-4}(D\dot D)^2\bigr)
 &=\frac{2^{n-1}\nu_{n-1}}{n(n+2)}
   \int_M\frac{f}{N}\bigl((\trg \dot g)^2+2|\dot g|_g^2\bigr)\dd\vol_g,
 \label{eq:velocity-result}\\
 \Wres\bigl(fN^{-1}|D|^{-n-4}D^2\dot D^2\bigr)&=0.\notag
\end{align}
\end{proposition}

\begin{proof}
The composition rule \eqref{eq:composition} and the order identity
$(-n-4)+1+1+1+1=-n$ give
\begin{align*}
 &\sigma_{-n}\bigl(fN^{-1}|D|^{-n-4}(D\dot D)^2\bigr)
 =fN^{-1}|\xi|_g^{-n-4}
       \sigma_1(D)\sigma_1(\dot D)\sigma_1(D)\sigma_1(\dot D).
\end{align*}
Indeed, every loss of symbol order or $\xi$-derivative makes the
degree strictly smaller than $-n$. Hence
Lemma~\ref{lem:velocity-traces} and
Lemma~\ref{lem:spherical-moments} yield
\begin{align*}
 \int_{|\xi|_g=1}
   \trE\sigma_{-n}\bigl(fN^{-1}|D|^{-n-4}(D\dot D)^2\bigr)                                   S_g
 &
 =\frac{2^{n-1}f}{N}
   \int_{|\xi|_g=1}
       \bigl[\dot g(\xi^{\sharp_g},\xi^{\sharp_g})\bigr]^2\dd S_g\\
 &
 =\frac{2^{n-1}\nu_{n-1}f}{n(n+2)N}
           \bigl((\trg\dot g)^2+2|\dot g|_g^2\bigr).
\end{align*}
Integration over $M$ proves the first identity.

For the second composition,
\[
 \sigma_2(\dot D^2)=\sigma_1(\dot D)^2=0,
 \qquad \operatorname{ord}(\dot D^2)\leq1.
\]
Consequently,
\[
 \operatorname{ord}
   \bigl(fN^{-1}|D|^{-n-4}D^2\dot D^2\bigr)
 \leq -n-4+2+1=-n-1,
\]
so its degree $-n$ symbol and residue both vanish.
\end{proof}

The first line of \eqref{eq:functional} is therefore
\begin{equation}
\label{eq:kinetic-result}
\begin{aligned}
 &\Wres\!\left[fN^{-1}|D|^{-n-4}
       \left(\frac{n(n+2)}8(D\dot D)^2
                   +\frac{n(n-4)}8D^2\dot D^2\right)\right]\\
 &
 =\frac{n(n+2)}8\,
   \frac{2^{n-1}\nu_{n-1}}{n(n+2)}
      \int_M\frac fN
            \bigl((\trg\dot g)^2+2|\dot g|_g^2\bigr)\dd\vol_g\\
 &
 =2^n\nu_{n-1}\int_M
       f\,\frac{(\trg\dot g)^2+2|\dot g|_g^2}{16N}\dd\vol_g.
\end{aligned}
\end{equation}
This identity holds for every smooth real $f$. When $f\geq0$,
its right-hand side is nonnegative.

\subsection{The weighted Hodge curvature term}

We express the traced Bochner potential in terms of $v_\Psi$
from \eqref{eq:vPsi}, using the standard decomposition for
Dirac-type operators \cite{BGV1992,Gilkey1995}.

\begin{lemma}[Bochner potential of a perturbed Hodge operator]
\label{lem:potential}
The Bochner decomposition of $D^2$ has a potential $V_\Psi$ satisfying
\begin{equation}
\label{eq:potential-trace}
 \trE V_\Psi=2^n\left(\frac14R_g+v_\Psi\right).
\end{equation}
\end{lemma}

\begin{proof}
Write $\nabla_i=\nabla^{E,g}_{\partial_i}$, with the induced
connection also used on endomorphisms. Applying the product rule
to a section of $E$ gives
\begin{align*}
 D^2&=(D_g^0)^2+\antic{D_g^0}{\Psi}+\Psi^2,\\
 \antic{D_g^0}{\Psi}
 &=\antic{c_g(dx^j)}{\Psi}\nabla_j+c_g(dx^j)(\nabla_j\Psi).
\end{align*}
For this proof set $B^j=\antic{c_g(dx^j)}{\Psi}$,
$B_i=g_{ij}B^j$, and $\nabla'_i=\nabla_i-B_i/2$.
The Bochner Laplacian is
$\nabla^*\nabla=-g^{ij}(\nabla_i\nabla_j-
\Gamma(g)^k{}_{ij}\nabla_k)$. Expanding the shifted connection
gives
\[
 (\nabla')^*\nabla'
 =\nabla^*\nabla+B^j\nabla_j
   +\frac12\nabla_jB^j-\frac14g_{ij}B^iB^j.
\]
Here $\nabla_jB^j$ includes the covariant derivative of the
contravariant index as well as that of the endomorphism.
Comparing with \eqref{eq:Weitzenbock} gives
\begin{equation}
\label{eq:potential-full}
 V_\Psi=\mathcal R_g+c_g(dx^j)(\nabla_j\Psi)+\Psi^2
             -\frac12\nabla_jB^j+\frac14g_{ij}B^iB^j.
\end{equation}
The Clifford action is parallel for the Levi--Civita connection.
Thus the product rule and cyclicity imply
\[
 \trE(\nabla_jB^j)
 =\trE\antic{c_g(dx^j)}{\nabla_j\Psi}
 =2\trE\bigl(c_g(dx^j)(\nabla_j\Psi)\bigr).
\]
Thus the two differentiated terms in \eqref{eq:potential-full}
cancel under the fiber trace at each point. Using
$\trE\mathcal R_g=2^{n-2}R_g$ now proves
\eqref{eq:potential-trace}. No integration by parts is involved
in this cancellation.
\end{proof}

The resulting weighted residue formula is the Hodge specialization
of \cite[Proposition~3.8]{BDSZ2026}. We state it in terms of
$v_\Psi$ and derive it from the preceding heat coefficient.

\begin{proposition}[Weighted perturbed Hodge KKW formula]
\label{prop:weighted-KKW}
For a smooth scalar function $f$,
\begin{equation}
\label{eq:weighted-KKW}
 \Wres(f|D|^{2-n})
 =2^n\nu_{n-1}\int_M f
 \left(-\frac{n-2}{24}R_g-\frac{n-2}{2}v_\Psi\right)
 \dd\vol_g.
\end{equation}
\end{proposition}

\begin{proof}
Apply Lemma~\ref{lem:heat-residue} to $A=D^2$.
Lemma~\ref{lem:potential} gives
\[
 \trE\left(\frac16R_g\Id_E-V_\Psi\right)
 =2^n\left(-\frac1{12}R_g-v_\Psi\right).
\]
Therefore Lemma~\ref{lem:heat-residue} gives
\begin{align*}
 \Wres(f|D|^{2-n})
 &=\frac{n-2}{2}\nu_{n-1}
   \int_M f\,2^n\left(-\frac1{12}R_g-v_\Psi\right)\dd\vol_g\\
 &=2^n\nu_{n-1}\int_M
       f\left(-\frac{n-2}{24}R_g-\frac{n-2}{2}v_\Psi\right)
          \dd\vol_g.
\end{align*}
\end{proof}

The contracted Clifford relation gives
$g_{ij}c_g(dx^i)c_g(dx^j)=-n\Id_E$.
By cyclicity, the two expressions for the perturbation term agree:
\begin{align*}
 g_{ij}\trE\bigl(\antic{c_g(dx^i)}{\Psi}
                         \antic{c_g(dx^j)}{\Psi}\bigr)
 &=2g_{ij}\trE\bigl(c_g(dx^i)\Psi c_g(dx^j)\Psi\bigr)
       -2n\trE\Psi^2,\\
 2^n v_\Psi
 &=\left(1-\tfrac n2\right)\trE\Psi^2
     +\tfrac12g_{ij}\trE\bigl(c_g(dx^i)\Psi c_g(dx^j)\Psi\bigr).
\end{align*}
For $\Psi=0$ and, respectively, for the weight $fN$, this yields
\begin{align*}
 \Wres(f|D_g^0|^{2-n})
 &=-\frac{(n-2)2^n}{24}\nu_{n-1}\int_M fR_g\dd\vol_g,\\
 \Wres(fN|D|^{2-n})
 &=2^n\nu_{n-1}\int_M
       fN\left(-\frac{n-2}{24}R_g-\frac{n-2}{2}v_\Psi\right)
          \dd\vol_g.
\end{align*}

\subsection{The lapse correction}

Write $A=D^2$ and $W=|D|^{-n-2}$. We use the convention
$[A,f]=Af-fA$ for a scalar multiplication operator $f$.
Since $A$ has scalar principal symbol, $[A,f]$ has order at
most one. Commuting $N^{-1}$ past the second copy of $A$ gives
\begin{align}
 N A N^{-1}A N W
 &=N A\bigl(A N^{-1}-[A,N^{-1}]\bigr)N W\notag\\
 &=N|D|^{2-n}-P W,
 \label{eq:lapse-split}
\end{align}
where
\[
 P=N A[A,N^{-1}]N.
\]
The equalities involving inverse powers are understood modulo
smoothing terms. The operator $P$ has order at most three.
This reduction isolates the weighted KKW term before any
lower-order symbols are calculated.

\begin{proposition}[Third-order symbol calculation]
\label{prop:third-order}
At each $x\in M$, the local residue density satisfies
\begin{align*}
 \operatorname{res}_x(PW)
 &:=\left(\int_{|\xi|_g=1}
       \trE\sigma_{-n}(PW)(x,\xi)\dd S_g(\xi)\right)\dd\vol_g(x)\\
 &=2^n\nu_{n-1}
 \left(\frac{n+4}{n}\Div_g\grad_g N
              -\frac4{nN}|\dd N|_g^2\right)(x)\dd\vol_g(x).
\end{align*}
Consequently,
\begin{equation}
\label{eq:third-order-integral}
 \Wres(PW)=-\frac{2^{n+2}\nu_{n-1}}n
                  \int_M N^{-1}|\dd N|_g^2\dd\vol_g.
\end{equation}
\end{proposition}

\begin{proof}
Fix a time and choose $g$-normal coordinates centered at $x_0$.
We retain the coordinate exterior frame and write
\[
 A=-g^{ij}\partial_i\partial_j+b^j\partial_j+v,
\]
where $b^j$ and $v$ are endomorphism-valued coefficients. Normal
coordinates give $g^{ij}(x_0)=\delta^{ij}$ and
$\partial_k g^{ij}(x_0)=0$, but do not in general give $b^j(x_0)=0$.
Indeed, expansion of $(D_g^0+\Psi)^2$ using
\eqref{eq:Weitzenbock} gives
\[
 b^j=-2g^{ij}\omega_i
      +g^{ik}\Gamma(g)^j{}_{ik}\Id_E
      +\antic{c_g(dx^j)}{\Psi}.
\]
The Levi--Civita coefficients $\omega_i$ and $\Gamma(g)^j{}_{ik}$
vanish at $x_0$, so
\[
 b^j(x_0)=\antic{c_g(dx^j)}{\Psi}(x_0),
 \qquad
 \sigma_1(A)(x_0,\xi)=\ii b^j(x_0)\xi_j.
\]
Put $\varphi=N^{-1}$. For a local section $\alpha$ of $E$,
the product rule gives
\begin{align*}
 [A,\varphi]\alpha
 &=A(\varphi\alpha)-\varphi A\alpha\\
 &=-2g^{ij}(\partial_i\varphi)\partial_j\alpha
   -g^{ij}(\partial_i\partial_j\varphi)\alpha
   +b^j(\partial_j\varphi)\alpha.
\end{align*}
Here $[v,\varphi]=0$, since $\varphi$ is scalar. Write
$[A,\varphi]N=u^j\partial_j\Id_E+v_0$. Applying this
operator to $\alpha$ and using
$\partial_j(N\alpha)=(\partial_jN)\alpha+N\partial_j\alpha$
gives
\begin{align*}
 [A,\varphi](N\alpha)
 ={}&-2Ng^{ij}(\partial_i\varphi)\partial_j\alpha+\left\{
   \bigl[-2g^{ij}(\partial_i\varphi)(\partial_jN)
          -Ng^{ij}\partial_i\partial_j\varphi\bigr]\Id_E
       +Nb^j\partial_j\varphi
   \right\}\alpha.
\end{align*}
Thus
\begin{align*}
 u^j&=-2Ng^{ij}\partial_i\varphi
       =2g^{ij}\frac{\partial_iN}{N},\\
 v_0&=\bigl[-2g^{ij}(\partial_i\varphi)(\partial_jN)
             -Ng^{ij}\partial_i\partial_j\varphi\bigr]\Id_E
       +Nb^j\partial_j\varphi.
\end{align*}

Set $N_i=\partial_iN(x_0)$ and
$N_{ij}=\partial_i\partial_jN(x_0)$. In the following
pointwise formulas, $N=N(x_0)$ and $b^j=b^j(x_0)$.
Differentiating $\varphi=N^{-1}$ gives
\[
 \partial_i\varphi(x_0)=-\frac{N_i}{N^2},
 \qquad
 \partial_i\partial_j\varphi(x_0)
 =-\frac{N_{ij}}{N^2}+\frac{2N_iN_j}{N^3}.
\]
In particular,
\[
 \left.Nb^j\partial_j\varphi\right|_{x_0}
 =Nb^j\left(-\frac{N_j}{N^2}\right)
 =-\frac{b^jN_j}{N}.
\]
The scalar gradient-square terms in $v_0$ cancel:
\begin{align*}
 v_0(x_0)
 &=\left[
       \frac{2}{N^2}\sum_iN_i^2
       +\frac1N\sum_iN_{ii}
       -\frac{2}{N^2}\sum_iN_i^2
     \right]\Id_E-\frac{b^jN_j}{N}\\
 &=\frac{\sum_iN_{ii}}{N}\Id_E-\frac{b^jN_j}{N}.
\end{align*}
To differentiate $u^j$, first use its expression near $x_0$:
\[
 \partial_i u^j
 =2(\partial_i g^{\ell j})\frac{\partial_\ell N}{N}
   +2g^{\ell j}\left(
      \frac{\partial_i\partial_\ell N}{N}
      -\frac{(\partial_iN)(\partial_\ell N)}{N^2}\right).
\]
Normal coordinates then give
\[
 u^j(x_0)=\frac{2N_j}{N},
 \qquad
 \partial_i u^j(x_0)
 =\frac{2N_{ij}}N-\frac{2N_iN_j}{N^2}.
\]

Write $\operatorname{Diff}^1(E)$ for differential operators on $E$
of order at most one. In the product
\[
 P=N(-g^{ik}\partial_i\partial_k+b^i\partial_i+v)
        (u^j\partial_j\Id_E+v_0),
\]
the required applications of the product rule are
\begin{align*}
 \partial_i\partial_k(u^j\partial_j\alpha)
 ={}&u^j\partial_i\partial_k\partial_j\alpha
       +(\partial_i u^j)\partial_k\partial_j\alpha+(\partial_k u^j)\partial_i\partial_j\alpha
       +(\partial_i\partial_k u^j)\partial_j\alpha,\\
 \partial_i\partial_k(v_0\alpha)
 ={}&v_0\partial_i\partial_k\alpha
       +(\partial_i v_0)\partial_k\alpha
       +(\partial_k v_0)\partial_i\alpha
       +(\partial_i\partial_k v_0)\alpha,\\
 b^i\partial_i(u^j\partial_j\alpha)
 ={}&b^iu^j\partial_i\partial_j\alpha
       +b^i(\partial_i u^j)\partial_j\alpha.
\end{align*}
Using $g^{ik}=g^{ki}$ and retaining orders three and two yields
\begin{align*}
 P\equiv{}&
   -Ng^{ik}u^j\partial_i\partial_k\partial_j\Id_E
   -2Ng^{ik}(\partial_i u^j)\partial_k\partial_j\Id_E\\
 &-Ng^{ik}v_0\partial_i\partial_k
   +Nb^iu^j\partial_i\partial_j
                 \pmod{\operatorname{Diff}^1(E)}.
\end{align*}
We now substitute $u^j$, $\partial_i u^j$, and $v_0$.
Writing $P|_{x_0}$ for coefficientwise evaluation in these
coordinates, with the differential operators left in place, gives
\begin{align*}
 \left.P\right|_{x_0}\equiv{}&
 -N\delta^{ik}\frac{2N_j}{N}
             \partial_i\partial_k\partial_j\Id_E\\
 &-2N\delta^{ik}\left(\frac{2N_{ij}}N
                 -\frac{2N_iN_j}{N^2}\right)
             \partial_k\partial_j\Id_E\\
 &-N\delta^{ik}\left[
       \frac{\sum_\ell N_{\ell\ell}}N\Id_E
           -\frac{b^jN_j}N\right]\partial_i\partial_k
   +Nb^i\frac{2N_j}N\partial_i\partial_j\\
 ={}&
 -2N_j\partial_i\partial_i\partial_j\Id_E
 -4\left(N_{ij}-\frac{N_iN_j}N\right)
                \partial_i\partial_j\Id_E\\
 &-\left(\sum_\ell N_{\ell\ell}\right)
                \partial_i\partial_i\Id_E
 +(b^jN_j)\partial_i\partial_i
 +2b^iN_j\partial_i\partial_j
                  \pmod{\operatorname{Diff}^1(E)}.
\end{align*}
The last two terms are the contributions of the first-order
coefficient of $A$. In particular, the negative $b^jN_j/N$
term in $v_0$ becomes positive after multiplication by $-N$.

Set $p_k=\sigma_k(P)(x_0,\xi)$. With
$\sigma(\partial_j)=\ii\xi_j$, the third-order term gives
$-2N_j(\ii\xi_i)^2(\ii\xi_j)
=2\ii|\xi|_g^2N_j\xi_j$, whereas each second-order term
acquires the factor $\ii^2=-1$. Hence
\begin{align*}
 p_3&=2\ii|\xi|_g^2N_j\xi_j\Id_E,\\
 p_2&=\left[
      4N_{ij}\xi_i\xi_j-\frac4N(N_i\xi_i)^2
           +\left(\sum_iN_{ii}\right)|\xi|_g^2
      \right]\Id_E-(b^jN_j)|\xi|_g^2
            -2(b^i\xi_i)(N_j\xi_j).
\end{align*}

Next let $r_{-2}$ and $r_{-3}$ be the first two symbols of a
parametrix of $A$. The terms of degrees zero and minus one
in $\sigma(A)\#\sigma(A^{-1})=\Id_E$ are
\begin{align*}
 |\xi|_g^2r_{-2}&=\Id_E,\\
 |\xi|_g^2r_{-3}+\sigma_1(A)r_{-2}
   +\sum_j\partial_{\xi_j}(|\xi|_g^2)D_{x_j}r_{-2}&=0.
\end{align*}
Consequently,
\begin{align*}
 r_{-2}&=|\xi|_g^{-2}\Id_E,\\
 r_{-3}
 &=-|\xi|_g^{-2}
   \left[\sigma_1(A)|\xi|_g^{-2}
      +\sum_j\partial_{\xi_j}(|\xi|_g^2)
                    D_{x_j}(|\xi|_g^{-2})\Id_E\right].
\end{align*}
At the normal-coordinate center,
\[
 D_{x_j}(|\xi|_g^{-2})(x_0,\xi)
 =-\frac1{\ii}|\xi|_g^{-4}
       (\partial_{x_j}g^{k\ell})(x_0)\xi_k\xi_\ell=0.
\]
Thus the derivative term vanishes, but the first-order symbol
of $A$ remains:
\[
 r_{-3}(x_0,\xi)
 =-|\xi|_g^{-2}(\ii b^j\xi_j)|\xi|_g^{-2}
 =-\ii b^j\xi_j|\xi|_g^{-4}.
\]

Since $W=A^{-(m+1)}$ modulo smoothing terms, its principal
symbol is the product of $m+1$ copies of $r_{-2}$.
A contribution one order lower contains either one $r_{-3}$
or one differentiated pair of principal symbols. At $x_0$
every contribution of the second kind contains
$\partial_x r_{-2}(x_0,\xi)=0$. The first kind gives
\[
 \sum_{\ell=0}^{m}r_{-2}^{\,\ell}r_{-3}r_{-2}^{\,m-\ell}
 =-(m+1)\ii b^j\xi_j|\xi|_g^{-2m-4},
\]
because $r_{-2}$ is scalar. Thus, for $q_k=\sigma_k(W)$,
\begin{equation}
\label{eq:W-symbols}
\begin{aligned}
 q_{-n-2}(x,\xi)&=|\xi|_g^{-n-2}\Id_E,\\
 q_{-n-3}(x_0,\xi)
   &=-\frac{n+2}{2}\ii b^j\xi_j|\xi|_g^{-n-4},\\
 \partial_{x_j}q_{-n-2}(x_0,\xi)&=0.
\end{aligned}
\end{equation}
The leading symbol is specified near $x_0$, so its derivative
is taken before evaluation at the normal-coordinate center.

The order condition for the degree $-n$ composition is
$(3-a)+(-n-2-\ell)-|\alpha|=-n$, or
$a+\ell+|\alpha|=1$. Consequently,
\begin{equation}
\label{eq:three-symbol-terms}
 \sigma_{-n}(PW)
 =p_2q_{-n-2}+p_3q_{-n-3}
       +\sum_j\partial_{\xi_j}p_3D_{x_j}q_{-n-2}.
\end{equation}
At $x_0$, the differentiated term in
\eqref{eq:three-symbol-terms} vanishes because
$D_{x_j}q_{-n-2}(x_0,\xi)=0$. The other two terms are
\begin{align*}
 p_2q_{-n-2}
 ={}&|\xi|_g^{-n-2}
  \left[4N_{ij}\xi_i\xi_j-\frac4N(N_i\xi_i)^2
       +\left(\sum_iN_{ii}\right)|\xi|_g^2\right]\Id_E\\
 &-(b^jN_j)|\xi|_g^{-n}
       -2(b^i\xi_i)(N_j\xi_j)|\xi|_g^{-n-2},\\
 p_3q_{-n-3}
 ={}&\bigl(2\ii|\xi|_g^2N_j\xi_j\bigr)
       \left(-\frac{n+2}{2}\ii b^i\xi_i|\xi|_g^{-n-4}\right)\\
 ={}&(n+2)(b^i\xi_i)(N_j\xi_j)|\xi|_g^{-n-2}.
\end{align*}
Here $N_j\xi_j$ is scalar, so it commutes with $b^i\xi_i$.
The coefficient of
$(b^i\xi_i)(N_j\xi_j)|\xi|_g^{-n-2}$ in their sum is
$-2+(n+2)=n$.
Combining it with $p_2q_{-n-2}$ yields
\begin{align*}
 \sigma_{-n}(PW)(x_0,\xi)
 ={}&|\xi|_g^{-n-2}
  \left[4N_{ij}\xi_i\xi_j-\frac4N(N_i\xi_i)^2
          +\left(\sum_iN_{ii}\right)|\xi|_g^2\right]\Id_E\\
 &-(b^jN_j)|\xi|_g^{-n}
       +n(b^i\xi_i)(N_j\xi_j)|\xi|_g^{-n-2}.
\end{align*}
On the unit sphere the two $b$-dependent terms cancel after
angular integration, even before taking the fiber trace:
\begin{align*}
 \int_{S^{n-1}}\!
 \left[-b^jN_j+n(b^i\xi_i)(N_j\xi_j)\right]\dd S
 &=-\nu_{n-1}b^jN_j
   +n b^iN_j\frac{\nu_{n-1}}n\delta_{ij}\\
 &=0.
\end{align*}
Using $\trE(\Id_E)=2^n$ and
\eqref{eq:sphere-moments} for the remaining scalar terms gives
\begin{align*}
 \int_{|\xi|_g=1}\trE\sigma_{-n}(PW)\dd S_g
 &=2^n\nu_{n-1}
    \left[\frac4n\sum_iN_{ii}
       -\frac4{nN}\sum_iN_i^2+\sum_iN_{ii}\right]\\
 &=2^n\nu_{n-1}
    \left[\frac{n+4}{n}\Div_g\grad_gN
       -\frac4{nN}|\dd N|_g^2\right].
\end{align*}
At a normal-coordinate center, $N_{ij}$ are the components of
$\Hess_gN$. This proves the local formula. Since $M$ is closed,
the divergence theorem gives
$\int_M\Div_g\grad_gN\dd\vol_g=0$, proving
\eqref{eq:third-order-integral}.
\end{proof}

This calculation uses only normal coordinates for the metric.
It imposes no normal-frame condition on the Bochner connection
of $A$. The first-order coefficients contribute to both
$p_2$ and $q_{-n-3}$; their cancellation occurs in the angular
integral, not by setting either coefficient to zero.

\begin{corollary}[Scalar-weighted lapse correction]
\label{cor:weighted-lapse}
For a smooth real function $f$ on a closed slice,
\[
 \Wres(fPW)=2^n\nu_{n-1}\int_M
 \left[-\frac{n+4}{n}\langle\dd f,\dd N\rangle_g
       -\frac{4f}{nN}|\dd N|_g^2\right]\dd\vol_g.
\]
\end{corollary}
\begin{proof}
Left multiplication gives
$\operatorname{res}_x(fPW)=f\operatorname{res}_x(PW)$.
Apply Proposition~\ref{prop:third-order} and integrate
$\Div_g(f\grad_gN)
=f\Div_g\grad_gN+\langle\dd f,\dd N\rangle_g$.
The integral of this divergence is zero on the closed manifold.
\end{proof}
The cyclic verification that follows concerns the unweighted
case $f=1$. For variable $f$, moving a differential operator
past $f$ produces a commutator; the preceding corollary is
the weighted formula used below.

The following identity is the Laplace-type version of the
commutator rearrangement used in
\cite[Section~7, preceding (8)]{Hawkins1997}. It provides a
verification of the integrated formula without the local symbol
calculation of Proposition~\ref{prop:third-order}.

\begin{proposition}[Cyclic verification of the lapse correction]
\label{prop:cyclic-lapse}
For any nonnegative self-adjoint Laplace-type operator $A$ on $E$
with principal symbol $|\xi|_g^2\Id_E$, and any positive scalar
$N$, put $W=A^{-(n+2)/2}$. Then
\begin{align}
 &\Wres(NA N^{-1}A N W)-\Wres(NA^2W)\notag\\
 &\hspace{2em}=
  -\Wres\bigl(N^{-1}[A,N]W[A,N]\bigr)
  =\frac{2^{n+2}\nu_{n-1}}n
           \int_M N^{-1}|\dd N|_g^2\dd\vol_g.
 \label{eq:cyclic-lapse}
\end{align}
\end{proposition}

\begin{proof}
Set $C=[A,N]$ in this proof. Expanding both commutators gives
\begin{align*}
 \Wres(N^{-1}CWC)
 ={}&\Wres(N^{-1}ANWAN)-\Wres(N^{-1}ANWNA)\\
    &-\Wres(AWAN)+\Wres(AWNA).
\end{align*}
Since $AW=WA$ modulo smoothing operators, cyclicity yields
\begin{align*}
 \Wres(AWAN)&=\Wres(NAWA)=\Wres(NA^2W),\\
 \Wres(AWNA)&=\Wres(NA^2W),\\
 \Wres(N^{-1}ANWAN)&=\Wres(ANWA)=\Wres(NA^2W),\\
 \Wres(N^{-1}ANWNA)&=\Wres(NA N^{-1}A N W).
\end{align*}
Substitution gives
\[
 \Wres(N^{-1}CWC)
 =\Wres(NA^2W)-\Wres(NA N^{-1}A N W).
\]

Write $N_i=\partial_iN$. The principal symbols satisfy
\[
 \sigma_1(C)=-2\ii g^{ij}N_i\xi_j\Id_E,
 \qquad \sigma_{-n-2}(W)=|\xi|_g^{-n-2}\Id_E.
\]
Since $\operatorname{ord}(N^{-1}CWC)\leq1-(n+2)+1=-n$,
\begin{align*}
 \sigma_{-n}(N^{-1}CWC)
 &=N^{-1}\sigma_1(C)\sigma_{-n-2}(W)\sigma_1(C)\\
 &=-4N^{-1}(g^{ij}N_i\xi_j)^2|\xi|_g^{-n-2}\Id_E.
\end{align*}
At a normal-coordinate center, Lemma~\ref{lem:spherical-moments}
therefore gives
\begin{align*}
 \int_{|\xi|_g=1}\trE\sigma_{-n}(N^{-1}CWC)\dd S_g
 &=-\frac{2^{n+2}}N\sum_{i,j}N_iN_j
                         \int_{S^{n-1}}\xi_i\xi_j\dd S\\
 &=-\frac{2^{n+2}\nu_{n-1}}{nN}\sum_iN_i^2
 =-\frac{2^{n+2}\nu_{n-1}}{nN}|\dd N|_g^2.
\end{align*}
Consequently,
\begin{align*}
 \Wres(NA N^{-1}A N W)-\Wres(NA^2W)
 &=-\Wres(N^{-1}CWC)\\
 &=\frac{2^{n+2}\nu_{n-1}}n
                     \int_M N^{-1}|\dd N|_g^2\dd\vol_g.
\end{align*}
For $A=D^2$, \eqref{eq:lapse-split} then recovers
\eqref{eq:third-order-integral}.
\end{proof}

The cyclic argument is an identity of integrated residues. It does
not equate the local densities of its two sides without a divergence
term. Proposition~\ref{prop:third-order} displays that divergence
explicitly. Both proofs also show that no connection or potential
term remains in the lapse correction: only the scalar principal
symbol of $A$ is used in \eqref{eq:cyclic-lapse}.

\subsection{The Hamiltonian-type formula}
\label{sec:torsion-main}

Combining the preceding three residue calculations gives the
following formula. We then apply it to the perturbation induced
by a change of metric connection, using the same coordinate
exterior frame.

\begin{theorem}
\label{thm:general-main}
Let $M^{2m}$ be closed and oriented, with $m\geq2$. Let $g_t$ be
a smooth family of Riemannian metrics, $N(t,x)>0$ a smooth scalar
function, $f(t,x)$ a smooth real function, and $\Psi_t$ a smooth
self-adjoint endomorphism of
$E=\Lambda^*T^*M\otimes\mathbb C$. With coefficient differentiation
on this fixed bundle, the functional of
Definition~\ref{def:functional} satisfies
\begin{equation}
\label{eq:general-main}
\begin{aligned}
 \mathscr L_t(f)=2^n\nu_{n-1}\int_M
 \Bigg[&f\left\{
 \frac{(\trg \dot g)^2+2|\dot g|_g^2}{16N}
 -\frac14NR_g-3Nv_\Psi
 -\frac{n-2}{4N}|\dd N|_g^2
 \right\}\\
 &-\frac{(n-2)(n+4)}{16}\langle\dd f,\dd N\rangle_g
 \Bigg]\dd\vol_g .
\end{aligned}
\end{equation}
All quantities on the right are evaluated at $t$, and $v_\Psi$ is
given by \eqref{eq:vPsi}.
\end{theorem}

\begin{proof}
The first line of \eqref{eq:functional} is
\eqref{eq:kinetic-result}. For this calculation set
\[
 B=\Wres(fN|D|^{2-n}),\qquad
 F=\Wres(fND^2N^{-1}D^2N|D|^{-n-2}).
\]
By \eqref{eq:lapse-split}, $F=B-\Wres(fPW)$. Therefore the
remaining two lines of \eqref{eq:functional} are
\begin{align*}
 &\left(\frac6{n-2}+\frac{n(n-2)}{16}\right)B
       -\frac{n(n-2)}{16}F\\
 &\quad=\frac6{n-2}B+\frac{n(n-2)}{16}\Wres(fPW)\\
 &\quad=2^n\nu_{n-1}\int_M
 f\left(-\frac14NR_g-3Nv_\Psi-\frac{n-2}{4N}|\dd N|_g^2\right)\dd\vol_g\\
 &\qquad-2^n\nu_{n-1}\frac{(n-2)(n+4)}{16}
       \int_M\langle\dd f,\dd N\rangle_g\dd\vol_g.
\end{align*}
The last equality uses Proposition~\ref{prop:weighted-KKW}
and Corollary~\ref{cor:weighted-lapse}. Adding the velocity term proves
\eqref{eq:general-main}.
\end{proof}

\begin{corollary}
\label{cor:dependence}
The functional $\mathscr L_t(f)$ is real for real $f$. Its dependence on the
endomorphism family is through $v_{\Psi_t}$ alone; neither
$\partial_t\Psi_t$ nor spatial derivatives of $\Psi_t$ occur.
If $\partial_tg_t=0$, $\Psi_t=0$, and $N$ is spatially constant,
then
\[
 \mathscr L_t(f)=-2^{n-2}\nu_{n-1}N(t)
                         \int_M fR_{g_t}\dd\vol_{g_t}.
\]
\end{corollary}

\begin{proof}
The adjoint calculation following \eqref{eq:vPsi} gives
$\overline{v_\Psi}=v_\Psi$. Since $f,N$ and the metric are real,
\eqref{eq:general-main} implies
\[
 \overline{\mathscr L_t(f)}=\mathscr L_t(f)\in\mathbb R.
\]
For the temporal derivative, \eqref{eq:velocity-symbol} gives
\[
 \sigma_1(\dot D)
 =-\ii(\partial_tg^{ij})\xi_j\iot(\partial_i),
 \qquad \sigma_1(\partial_t\Psi)=0.
\]
Only this first-order symbol enters the velocity residues in
Proposition~\ref{prop:velocity}. For spatial derivatives, use
$\nabla_j=\nabla^{E,g}_{\partial_j}$ and its induced action on
endomorphism-valued tensors. Lemma~\ref{lem:potential} gives
\begin{align*}
 &\trE\!\left(c_g(dx^j)(\nabla_j\Psi)
       -\frac12\nabla_j\antic{c_g(dx^j)}{\Psi}\right)\\
 &=\trE\bigl(c_g(dx^j)(\nabla_j\Psi)\bigr)
       -\frac12\trE\antic{c_g(dx^j)}{\nabla_j\Psi}=0.
\end{align*}
Thus the remaining perturbation term is the algebraic
contraction $v_\Psi$ in \eqref{eq:vPsi}.
Under the additional hypotheses,
\[
 \partial_tg_t=0,\quad\Psi_t=0,\quad N=N(t)
 \quad\Longrightarrow\quad
 \dot g=0,\quad v_\Psi=0,\quad \dd N=0.
\]
Hence
\begin{align*}
 \mathscr L_t(f)
 &=2^n\nu_{n-1}\int_M f\left(-\frac14N(t)R_{g_t}\right)
                                      \dd\vol_{g_t}\\
 &=-2^{n-2}\nu_{n-1}N(t)\int_M fR_{g_t}\dd\vol_{g_t}.
\end{align*}
\end{proof}

We finally record how a metric connection enters this formula.
This is the coordinate exterior representation of the induced
connection, as in \eqref{eq:exterior-connection}; related Hodge
operators with torsion are considered in
\cite[Section~3.2]{BDSZ2026}.

\begin{proposition}[Coordinate expression for a connection perturbation]
\label{prop:coordinate-perturbation}
Let $\widetilde\nabla$ be a metric connection on $TM$, and define
the tensor $\mathcal A$ by
$\widetilde\nabla_XY=\nabla^g_XY+\mathcal A(X,Y)$.
Write $\mathcal A(\partial_i,\partial_b)
=\mathcal A^a{}_{ib}\partial_a$. Its induced connection on
$E$ is
\begin{equation}
\label{eq:coordinate-metric-connection}
 \widetilde\nabla^E_{\partial_i}
 =\partial_i-
   \bigl(\Gamma(g)^a{}_{ib}+\mathcal A^a{}_{ib}\bigr)
          \eps(dx^b)\iot(\partial_a).
\end{equation}
Consequently the induced Dirac-type operator satisfies
\begin{equation}
\label{eq:coordinate-connection-Psi}
 \widetilde D
 :=c_g(dx^i)\widetilde\nabla^E_{\partial_i}
 =D_g^0+\Psi_{\mathcal A},
 \qquad
 \Psi_{\mathcal A}
 =-\mathcal A^a{}_{ib}c_g(dx^i)\eps(dx^b)\iot(\partial_a).
\end{equation}
It is formally self-adjoint if and only if
$\sum_i\mathcal A^i{}_{ib}=0$ for every $b$.
For a smooth family satisfying this condition,
Theorem~\ref{thm:general-main} applies with
$\Psi=\Psi_{\mathcal A}$.
\end{proposition}

\begin{proof}
The dual connection on coordinate one-forms is determined by
\begin{align*}
 0&=\partial_i\bigl(dx^a(\partial_b)\bigr)\\
  &=(\widetilde\nabla_{\partial_i}dx^a)(\partial_b)
     +dx^a(\widetilde\nabla_{\partial_i}\partial_b).
\end{align*}
Thus $\widetilde\nabla_{\partial_i}dx^a
=-(\Gamma(g)^a{}_{ib}+\mathcal A^a{}_{ib})dx^b$.
The induced connection acts on exterior products by the
Leibniz rule. Since $\eps(dx^b)\iot(\partial_a)$ replaces
each factor $dx^a$ by $dx^b$, this proves
\eqref{eq:coordinate-metric-connection} on every exterior degree.
Subtracting \eqref{eq:exterior-connection} and multiplying by
$c_g(dx^i)$ gives \eqref{eq:coordinate-connection-Psi}.
In particular, the change of connection contributes only a
zero-order term; the Clifford principal symbol is unchanged.

For the adjoint assertion, let
$K_i=\widetilde\nabla^E_{\partial_i}-\nabla^{E,g}_{\partial_i}$
within this proof. Metric compatibility gives
$K_i^*=-K_i$. Compatibility with Clifford multiplication gives
\[
 [K_i,c_g(dx^j)]=-\mathcal A^j{}_{ib}c_g(dx^b).
\]
It follows from $c_g(dx^i)^*=-c_g(dx^i)$ that
\begin{align*}
 \Psi_{\mathcal A}^*-\Psi_{\mathcal A}
 &=\sum_i\bigl(K_i c_g(dx^i)-c_g(dx^i)K_i\bigr)\\
 &=-\sum_{i,b}\mathcal A^i{}_{ib}c_g(dx^b).
\end{align*}
Clifford multiplication by a real covector is zero only if
that covector is zero, because its square is minus its
squared norm times $\Id_E$. Since $D_g^0$ is formally
self-adjoint, the stated criterion follows. Under this
condition, \eqref{eq:coordinate-connection-Psi} satisfies
the hypotheses of Theorem~\ref{thm:general-main}.
\end{proof}

The trace condition holds, in particular, when
$g(\mathcal A(X,Y),Z)$ is a three-form: contraction of its
first and third arguments is then zero. For a general
metric connection it is a separate condition and does
not follow from metric compatibility alone.
The operator $\widetilde D$ denotes the induced perturbed
operator; only when the perturbation vanishes is it
identified here with the unperturbed $d+\delta_g$.

\section{Boundary residues and the weighted functional}
\label{sec:boundary}

We compute the two boundary residues associated with the weighted
curvature term and the lapse commutator. After specifying their
factorizations, we derive the boundary symbols and evaluate the terms
in the noncommutative residue formula. The weighted curvature residue
has five possible contributions, while the commutator residue has
only one. Combining these calculations with the local formulas of
Section~\ref{sec:residue-calculations} gives the boundary formula.

\subsection{Boundary geometry and the residue formula}

Let $M^n$ be compact and oriented, with smooth boundary and
$n=2m\ge4$. On a fixed collar
$\partial M\times[0,\epsilon)$, assume that
\begin{equation}
\label{eq:boundary-collar}
 g_t=h_t(x_n)^{-1}g_t^{\partial M}+dx_n^2,
 \qquad h_t>0,\qquad h_t(0)=1.
\end{equation}
Here $x_n$ increases into $M$, the metric $g_t^{\partial M}$ is
pulled back from $\partial M$, and all coefficients are smooth in
$t$ and up to $x_n=0$. The scalar function $h_t$ is distinct from
the metric variation $\dot g_t=\partial_tg_t$. We allow both the
positive lapse $N$ and the real test function $f$ to depend on all
spatial variables. The self-adjoint endomorphism $\Psi_t$ is smooth
up to the boundary. We use the fixed coordinate exterior bundle,
with $D=D_{g_t}^0+\Psi_t$ and $A=D^2$, as in Section~\ref{sec:preliminaries}.
Unless a temporal derivative is displayed, $t$ is fixed and suppressed.

Extend the coefficients smoothly across the boundary and use
classical parametrices of the extended elliptic operator $D$.
The integer negative powers below are understood modulo smoothing
operators; in particular, $D^{-n-2}=|D|^{-n-2}$ in this sense.
These are not inverses of an operator with a boundary condition.
Their residue symbols depend only on the coefficient jets on $M$.

For an operator $T$ satisfying the transmission condition, put
$\pi^+T=r^+Te^+$, where $e^+$ is extension by zero and $r^+$
is restriction to the interior. The differential operator $D$,
its integer-order parametrices, and their compositions satisfy
this condition. For the negative-order factors used below, the
noncommutative residue of \cite{FGLS1996} splits as
\begin{equation}
\label{eq:boundary-split}
 \Wresb[\pi^+U\circ\pi^+V]
 =\int_M\operatorname{res}_x(UV)
   +\int_{\partial M}\Phi(U,V)\dd\vol_{g^{\partial M}},
\end{equation}
where $\operatorname{res}_x$ denotes the interior density with
the normalization \eqref{eq:Wres}. Thus $\Phi(U,V)$ is the
scalar coefficient of the boundary density.

Fix $x_0\in\partial M$ and choose $g^{\partial M}$-normal
coordinates $x'=(x^1,\ldots,x^{n-1})$ centered there. Write
$\xi=\xi'+\xi_n dx_n$. At $(x_0,0)$ the boundary coefficient is
\begin{equation}
\label{eq:FGLS-new}
\begin{aligned}
 \Phi(U,V)={}&
 \int_{|\xi'|_{g^{\partial M}}=1}\int_{\mathbb R}
 \sum_{\substack{j,k\ge0,\ \alpha\in\mathbb N_0^{n-1}\\
 r+\ell-j-k-|\alpha|-1=-n}}
 \frac{(-\ii)^{|\alpha|+j+k+1}}{\alpha!(j+k+1)!}
 \\
 &\quad\cdot\trE\left[
 \partial_{x_n}^{j}\partial_{\xi'}^{\alpha}\partial_{\xi_n}^{k}
       \sigma_r^+(U)\,
 \partial_{x'}^{\alpha}\partial_{\xi_n}^{j+1}\partial_{x_n}^{k}
       \sigma_{\ell}(V)
 \right]\dd\xi_n\dd S_{g^{\partial M}}(\xi').
\end{aligned}
\end{equation}
The indices satisfy $r\le\operatorname{ord}U$ and
$\ell\le\operatorname{ord}V$. Here
$\sigma_r^+(U)=\pi_{\xi_n}^+\sigma_r(U)$, where the normal
Hardy projection retains the principal parts at the upper-half-plane
poles of the rational symbol. This projection on symbols is
distinguished from the operator truncation $\pi^+T$.
All base derivatives are taken before evaluation at $(x_0,0)$,
and all tangential covariable derivatives before restriction to
$|\xi'|_{g^{\partial M}}=1$.
The convention is the symbol product \eqref{eq:composition};
see also \cite{Wang2007,WangWang2025Torsion}.

Recall $P=NA[A,N^{-1}]N$. We consider
\begin{align}
 \mathcal B_t(f,N)
 &:=\Wresb[\pi^+(fND^{-1})\circ\pi^+(D^{-n+3})],
 \label{eq:boundary-B}\\
 \mathcal C_t(f,N)
 &:=\Wresb[\pi^+(fPD^{-5})\circ\pi^+(D^{-n+3})].
 \label{eq:boundary-C}
\end{align}
The left factors have orders at most $-1$ and $-2$, respectively;
the right factor has order $3-n$. Their interior products are
$fN|D|^{2-n}$ and $fP|D|^{-n-2}$. These factorizations are part
of the definitions of the boundary functionals.

The identity \eqref{eq:lapse-split} implies
\begin{equation}
\label{eq:boundary-third-identity}
 fNAN^{-1}AN D^{-5}=fND^{-1}-fPD^{-5}
 \quad\bmod\Psi^{-\infty}.
\end{equation}
By linearity, the residue with left factor
$fNAN^{-1}AN D^{-5}$ and right factor $D^{-n+3}$ is
$\mathcal B_t(f,N)-\mathcal C_t(f,N)$.

\subsection{Boundary symbols and normal integrations}

Throughout this subsection all coefficients are evaluated at
$(x_0,0)$ after differentiation. Set $\eta=h'(0)$, and let
$a,b<n$ denote tangential indices. Before restriction to the
unit tangential sphere,
\[
 |\xi|_g^2=h(x_n)|\xi'|_{g^{\partial M}}^2+\xi_n^2,
 \qquad
 |\xi'|_{g^{\partial M}}^2=(g^{\partial M})^{ab}\xi_a\xi_b.
\]
The coordinate exterior frame is used in all symbol calculations.

\begin{lemma}[Connection and first-order symbols]
\label{lem:boundary-symbols-new}
The nonzero Christoffel coefficients at the chosen point are
\[
 \Gamma^n{}_{ab}=\frac\eta2\delta_{ab},\qquad
 \Gamma^a{}_{bn}=\Gamma^a{}_{nb}=-\frac\eta2\delta^a_b.
\]
The zero-order symbol of $D$ is
\begin{equation}
\label{eq:boundary-pzero}
 p_0:=\sigma_0(D)
 =\eta\left(\frac{n-1}{2}\Id_E
       -\sum_{a<n}\eps(dx^a)\iot(\partial_a)\right)
          \iot(\partial_n)+\Psi,
\end{equation}
and the first-order symbol of $A=D^2$ is
\begin{equation}
\label{eq:boundary-Aone}
\begin{aligned}
 \sigma_1(A)={}&\ii\eta\sum_{a<n}\xi_a
       \bigl(\eps(dx^a)\iot(\partial_n)
             -\eps(dx_n)\iot(\partial_a)\bigr)\\
 &+\ii\eta\xi_n\left(\frac{n-1}{2}\Id_E
                  -\sum_{a<n}\eps(dx^a)\iot(\partial_a)\right)
       +\ii\{c_g(\xi),\Psi\}.
\end{aligned}
\end{equation}
Furthermore,
$\partial_{x_a}|\xi|_g^2=\partial_{x_a}c_g(\xi)=0$ for $a<n$,
while
\[
 \partial_{x_n}|\xi|_g^2=\eta|\xi'|_{g^{\partial M}}^2,
 \qquad
 \partial_{x_n}c_g(\xi)
       =-\eta\sum_{a<n}\xi_a\iot(\partial_a).
\]
\end{lemma}

\begin{proof}
At the boundary center, $g_{ab}=\delta_{ab}$,
$\partial_{x_n}g_{ab}=-\eta\delta_{ab}$,
$g_{an}=0$, and $g_{nn}=1$. Substitution in the Christoffel
formula gives
$\Gamma^n{}_{ab}=-\tfrac12\partial_{x_n}g_{ab}$ and
$\Gamma^a{}_{bn}=\tfrac12g^{ac}\partial_{x_n}g_{bc}$;
the remaining coefficients vanish. By
\eqref{eq:exterior-connection}, the induced connection matrices are
\[
 \omega_a=\frac\eta2
   \bigl(\eps(dx_n)\iot(\partial_a)
              -\eps(dx^a)\iot(\partial_n)\bigr),\qquad
 \omega_n=\frac\eta2\sum_{a<n}\eps(dx^a)\iot(\partial_a).
\]
The zero-order part of $D_g^0$ is $\sum_i c_g(dx^i)\omega_i$.
For each $a<n$, the anticommutation relations \eqref{eq:CAR} give
\begin{align*}
 &c_g(dx^a)\bigl(\eps(dx_n)\iot(\partial_a)
                    -\eps(dx^a)\iot(\partial_n)\bigr)
       +c_g(dx_n)\eps(dx^a)\iot(\partial_a)\\
 &\qquad
   =\bigl(\Id_E-2\eps(dx^a)\iot(\partial_a)\bigr)
                       \iot(\partial_n).
\end{align*}
Summing and adding $\Psi$ proves \eqref{eq:boundary-pzero}.
The first-order coefficient of $(D_g^0)^2$ is
$-2g^{ij}\omega_i+g^{ab}\Gamma^j{}_{ab}\Id_E$.
The additional terms $D_g^0\Psi+\Psi D_g^0$ contribute
$\{c_g(dx^j),\Psi\}$ to the coefficient of $\partial_j$,
which proves \eqref{eq:boundary-Aone}.
Finally, differentiating
$c_g(\xi)=\eps(\xi)-g^{ij}\xi_i\iot(\partial_j)$
and the displayed quadratic symbol proves the derivative identities.
\end{proof}

Restrict now to $|\xi'|_{g^{\partial M}}=1$, and put $z=\xi_n$.
For the following pointwise formulas only, write
$c=c_g(\xi'+zdx_n)$. Let
$a_{-1}=\sigma_{-1}(D^{-1})$ and
$a_{-2}=\sigma_{-2}(D^{-1})$.
The first two parametrix equations are
\begin{align*}
 (\ii c_g(\xi))a_{-1}&=\Id_E,\\
 (\ii c_g(\xi))a_{-2}+p_0a_{-1}
       +\sum_j\partial_{\xi_j}(\ii c_g(\xi))D_{x_j}a_{-1}&=0.
\end{align*}
Since $c_g(\xi)^2=-|\xi|_g^2\Id_E$,
$a_{-1}=\ii c_g(\xi)|\xi|_g^{-2}$ before restriction. Thus
\[
 \partial_{x_n}a_{-1}
 =-\ii\eta\left[
    \frac{\sum_{a<n}\xi_a\iot(\partial_a)}{1+z^2}
                  +\frac{c}{(1+z^2)^2}\right].
\]
Only this base derivative survives in the second parametrix equation.
Using $D_{x_j}=-\ii\partial_{x_j}$ gives
\begin{align}
 a_{-1}&=\frac{\ii c}{1+z^2},
 \label{eq:boundary-inverse}\\
 a_{-2}
 &=\frac{cp_0c}{(1+z^2)^2}
   -\frac{\eta c\,c_g(dx_n)\sum_{a<n}\xi_a\iot(\partial_a)}
                {(1+z^2)^2}
   -\frac{\eta c\,c_g(dx_n)c}{(1+z^2)^3}.\notag
\end{align}

For the right factor set
$b_{3-n}=\sigma_{3-n}(D^{-n+3})$ and
$b_{2-n}=\sigma_{2-n}(D^{-n+3})$.
For any integer $k\ge1$, the first two symbols of $A^{-k}$ are
\begin{align*}
 \sigma_{-2k}(A^{-k})&=|\xi|_g^{-2k}\Id_E,\\
 \sigma_{-2k-1}(A^{-k})(x_0,0,\xi)
 &=-k\sigma_1(A)(1+z^2)^{-k-1}
       -\ii k(k+1)\eta z(1+z^2)^{-k-2}\Id_E.
\end{align*}
For $k=1$ the second identity follows from
\[
 \sigma_{-3}(A^{-1})
 =-|\xi|_g^{-2}\left[
       \sigma_1(A)|\xi|_g^{-2}
       +\sum_j\partial_{\xi_j}|\xi|_g^2
                             D_{x_j}(|\xi|_g^{-2})\Id_E\right].
\]
At the boundary, $D_{x_n}(|\xi|_g^{-2})
=\ii\eta(1+z^2)^{-2}$. In passing from $A^{-k}$ to
$A^{-k}A^{-1}$, the three subleading contributions have
coefficients $k(k+1)$, $2$, and $2k$, whose sum is
$(k+1)(k+2)$. This proves the identity by induction.
Now $D^{-n+3}=A^{-(m-1)}D$ modulo smoothing operators, so
\begin{align*}
 b_{2-n}={}&\sigma_{-n+2}(A^{-(m-1)})p_0
       +\sigma_{-n+1}(A^{-(m-1)})\,\ii c+\partial_z\sigma_{-n+2}(A^{-(m-1)})
                        D_{x_n}(\ii c_g(\xi)).
\end{align*}
Consequently,
\begin{align}
 b_{3-n}&=\frac{\ii c}{(1+z^2)^{m-1}},
 \label{eq:boundary-right-symbols}\\
 b_{2-n}
 &=\frac{p_0}{(1+z^2)^{m-1}}
   -\frac{\ii(m-1)\sigma_1(A)c}{(1+z^2)^m}
   +\frac{m(m-1)\eta zc}{(1+z^2)^{m+1}}\notag\\
 &\quad+\frac{2(m-1)\eta z
                         \sum_{a<n}\xi_a\iot(\partial_a)}{(1+z^2)^m}.
 \notag
\end{align}

\paragraph{The normal projection and the exterior traces.}
Partial fractions give
\begin{equation}
\label{eq:boundary-projections}
 \pi_z^+\frac1{1+z^2}=\frac1{2\ii(z-\ii)},\qquad
 \pi_z^+\frac z{1+z^2}=\frac1{2(z-\ii)},\qquad
 a_{-1}^+=\frac{c_g(\xi')+\ii c_g(dx_n)}{2(z-\ii)}.
\end{equation}
We also use
\begin{align*}
 \pi_z^+\frac1{(1+z^2)^2}
   &=-\frac1{4(z-\ii)^2}+\frac1{4\ii(z-\ii)},\\
 \pi_z^+\frac{1-z^2}{(1+z^2)^2}
   &=-\frac1{2(z-\ii)^2},\qquad
 \pi_z^+\frac z{(1+z^2)^2}=\frac1{4\ii(z-\ii)^2}.
\end{align*}
For a pole of order $q$ at $\ii$, these formulas follow from
\[
 \pi_z^+F(z)=
 \sum_{j=1}^{q}\frac1{(q-j)!}
 \left[\frac{d^{q-j}}{dz^{q-j}}
                    \bigl((z-\ii)^qF(z)\bigr)\right]_{z=\ii}
 (z-\ii)^{-j}.
\]
All symbols projected below are proper rational functions with
poles only at $\pm\ii$. Projection commutes with their base
derivatives, taken before evaluation at the boundary.

The exterior traces needed below are
\begin{align*}
 \trE\bigl(c_g(\alpha)c_g(\beta)\bigr)
   &=-2^ng^{-1}(\alpha,\beta),\\
 \trE\bigl(\eps(dx^a)\iot(\partial_b)\bigr)
   &=2^{n-1}\delta^a_b,\\
 \trE\bigl(\eps(dx^a)\iot(\partial_b)
                   \eps(dx^c)\iot(\partial_d)\bigr)
   &=2^{n-2}(\delta^a_b\delta^c_d+\delta^a_d\delta^c_b).
\end{align*}
They follow from \eqref{eq:CAR} and cyclicity of the ordinary
trace. Both Clifford squares have trace $-2^n$, whereas
$\trE(c_g(\xi')c_g(dx_n))=0$. Also,
\[
 \trE\left[\left(\sum_{a<n}\xi_a\iot(\partial_a)\right)
                                   c_g(\xi')\right]=2^{n-1},
 \qquad
 \trE\left[\left(\sum_{a<n}\xi_a\iot(\partial_a)\right)
                                   c_g(dx_n)\right]=0.
\]
Inserting these contractions in the derivative of $a_{-1}$ yields
\begin{align*}
 \trE\bigl((\partial_{x_n}a_{-1})^+c_g(\xi')\bigr)
     &=-\frac{\ii 2^{n-2}\eta}{(z-\ii)^2},\\
 \trE\bigl((\partial_{x_n}a_{-1})^+c_g(dx_n)\bigr)
     &=\frac{2^{n-2}\eta}{(z-\ii)^2}.
\end{align*}
For example, before projection the first trace is
$\ii 2^{n-1}\eta(1-z^2)/(1+z^2)^2$.

For subsequent substitutions it is useful to record the
unperturbed subleading traces. The notation $\left.\vphantom{a}T
\right|_{\Psi=0}$ means setting $\Psi$ equal to zero in the
displayed symbol, without changing the metric. Direct contraction gives
\begin{align}
 \trE\bigl(\left.a_{-2}\right|_{\Psi=0}c_g(\xi')\bigr)
    &=\frac{2^{n-1}\eta z(z^2-3)}{(1+z^2)^3},\notag\\
 \trE\bigl(\left.a_{-2}\right|_{\Psi=0}c_g(dx_n)\bigr)
    &=\frac{2^{n-1}\eta(1-3z^2)}{(1+z^2)^3},\notag\\
 \trE\bigl(c_g(\xi')\left.b_{2-n}\right|_{\Psi=0}\bigr)
    &=\frac{2^{n-1}\eta(m-1)z[z^2-(2m-1)]}{(1+z^2)^{m+1}},\notag\\
 \trE\bigl(c_g(dx_n)\left.b_{2-n}\right|_{\Psi=0}\bigr)
    &=\frac{2^{n-1}\eta(m-1)[1-(2m-1)z^2]}{(1+z^2)^{m+1}}.
 \label{eq:boundary-subleading-traces}
\end{align}
Here are the contractions giving these formulas. The two traces
$\trE((p_0-\Psi)c_g(\xi'))$ and
$\trE((p_0-\Psi)c_g(dx_n))$ vanish. Also,
\begin{align*}
 \trE\left[c\,c_g(dx_n)
       \left(\sum_{a<n}\xi_a\iot(\partial_a)\right)c_g(\xi')\right]
     &=-2^{n-1}z,\\
 \trE\left[c\,c_g(dx_n)
       \left(\sum_{a<n}\xi_a\iot(\partial_a)\right)c_g(dx_n)\right]
     &=2^{n-1},\\
 c\,c_g(dx_n)c&=(1-z^2)c_g(dx_n)-2z c_g(\xi').
\end{align*}
These identities prove the first two lines of
\eqref{eq:boundary-subleading-traces} by
\eqref{eq:boundary-inverse}. For the last two lines, the
nonzero contractions in the first line of
\eqref{eq:boundary-Aone} are determined by
\[
 \trE\left[
 \sum_{a<n}\xi_a\bigl(\eps(dx^a)\iot(\partial_n)
          -\eps(dx_n)\iot(\partial_a)\bigr)
                        c_g(dx_n)c_g(\xi')\right]=-2^{n-1}.
\]
The reversed Clifford product has the opposite trace. The
contractions of the second line's metric term with
$c_g(\xi')^2$, $c_g(dx_n)^2$, or
$c_g(\xi')c_g(dx_n)$ vanish. Substitution in
\eqref{eq:boundary-right-symbols} gives the stated expressions.
In particular, projection of the first two lines gives
\begin{equation}
\label{eq:boundary-subleading-projected}
 \trE\bigl((\left.a_{-2}\right|_{\Psi=0})^+c_g(\xi')\bigr)
       =\frac{2^{n-2}\eta}{(z-\ii)^3},\qquad
 \trE\bigl((\left.a_{-2}\right|_{\Psi=0})^+c_g(dx_n)\bigr)
       =\frac{\ii 2^{n-2}\eta}{(z-\ii)^3}.
\end{equation}

\paragraph{The scalar normal integrals.}
All remaining normal integrations reduce to
\begin{equation}
\label{eq:boundary-beta}
 \int_{\mathbb R}\frac{z^{2j}}{(1+z^2)^p}\dd z
 =\frac{\Gamma(j+\tfrac12)\Gamma(p-j-\tfrac12)}{\Gamma(p)},
 \qquad j\in\mathbb N_0,\quad p>j+\tfrac12.
\end{equation}
Indeed, the substitutions $y=z^2$ and $y=s/(1-s)$ give
\[
 \int_{\mathbb R}\frac{z^{2j}}{(1+z^2)^p}\dd z
 =\int_0^\infty\frac{y^{j-1/2}}{(1+y)^p}\dd y
 =\int_0^1s^{j-1/2}(1-s)^{p-j-3/2}\dd s,
\]
which is the beta integral. Odd integrands have zero integral.
For $m\ge2$, put
$a_m=\pi(2m-2)!/[2^{2m-1}(m-2)!m!]$. Then
\begin{equation}
\label{eq:boundary-am}
 \int_{\mathbb R}\frac{\dd z}{(1+z^2)^{m+1}}
       =\frac{2m-1}{m-1}a_m,
 \qquad
 \nu_{n-2}a_m=\frac{n-2}{2n}\nu_{n-1}.
\end{equation}
The second identity follows by substituting
$\nu_j=2\pi^{(j+1)/2}/\Gamma((j+1)/2)$ and $n=2m$.
The four integrals used in the calculations are therefore
\begin{align*}
 \int_{\mathbb R}\frac{z^2}{(1+z^2)^{m+1}}\dd z
       &=\frac{a_m}{m-1},\\
 \int_{\mathbb R}\frac{\dd z}{(1+z^2)^{m+2}}
       &=\frac{(2m+1)(2m-1)}{2(m+1)(m-1)}a_m,\\
 \int_{\mathbb R}\frac{z^2}{(1+z^2)^{m+2}}\dd z
       &=\frac{2m-1}{2(m+1)(m-1)}a_m.
\end{align*}
Together with the first identity in \eqref{eq:boundary-am},
these follow from \eqref{eq:boundary-beta} and
$\Gamma(s+1)=s\Gamma(s)$.

\subsection{The weighted KKW residue}

For this calculation put $w=fN$, and apply
\eqref{eq:FGLS-new} with $U=wD^{-1}$ and $V=D^{-n+3}$.
Writing $r=-1-a$ and $\ell=3-n-b$, where $a,b\ge0$,
the order condition is
\[
 (-1-a)+(3-n-b)-j-k-|\alpha|-1=-n
 \quad\Longleftrightarrow\quad a+b+j+k+|\alpha|=1.
\]
Thus there are five cases. Let $\Phi_1,\ldots,\Phi_5$
denote their contributions, including both the normal integral
and the tangential sphere integral, so that
\[
 \Phi(wD^{-1},D^{-n+3})=\Phi_1+\Phi_2+\Phi_3+\Phi_4+\Phi_5.
\]
In each case the parameters in \eqref{eq:FGLS-new} are
specified before substitution. We write $\dd S(\xi')$ for
$\dd S_{g^{\partial M}}(\xi')$ at the chosen point.
Terms odd in $z$ or $\xi'$ integrate to zero over
$\mathbb R$ or the unit tangential sphere, respectively.
When such terms are omitted below, the resulting equality
is an equality of integrals, not of the original integrands.

\paragraph{Case (a)(I).}
Here $r=-1$, $\ell=3-n$, $j=k=0$, and $|\alpha|=1$.
The coefficient in \eqref{eq:FGLS-new} is $(-\ii)^2=-1$.
Consequently,
\begin{align*}
 \Phi_1
 &=-\sum_{a<n}\int_{|\xi'|_{g^{\partial M}}=1}\int_{\mathbb R}
   \trE\left[
      \partial_{\xi_a}\sigma_{-1}^+(wD^{-1})\,
      \partial_{x_a}\partial_z\sigma_{3-n}(D^{-n+3})
         \right]\dd z\dd S(\xi')\\
 &=-w\sum_{a<n}\int_{|\xi'|_{g^{\partial M}}=1}\int_{\mathbb R}
   \trE\left[(\partial_{\xi_a}a_{-1}^+)
                          \partial_{x_a}\partial_zb_{3-n}\right]
                       \dd z\dd S(\xi').
\end{align*}
Before evaluation at the boundary,
\[
 \partial_{x_a}b_{3-n}
 =\ii(\partial_{x_a}c_g(\xi))|\xi|_g^{-n+2}
   -\ii(m-1)c_g(\xi)|\xi|_g^{-n}
                                      \partial_{x_a}|\xi|_g^2.
\]
Both terms vanish at $(x_0,0)$ by
Lemma~\ref{lem:boundary-symbols-new}, for every nonzero
$\xi$. Their $z$-derivatives therefore vanish as well, and
$\Phi_1=0$.

\paragraph{Case (a)(II).}
Take $r=-1$, $\ell=3-n$, $j=1$, and $k=|\alpha|=0$.
The coefficient in \eqref{eq:FGLS-new} is
$(-\ii)^2/2!=-1/2$. The product rule gives
\begin{align*}
 \Phi_2
 &=-\frac12\int_{|\xi'|_{g^{\partial M}}=1}\int_{\mathbb R}
   \trE\left[
      \partial_{x_n}\sigma_{-1}^+(wD^{-1})\,
      \partial_z^2\sigma_{3-n}(D^{-n+3})\right]\dd z\dd S(\xi')\\
 &=-\frac12\int_{|\xi'|_{g^{\partial M}}=1}\int_{\mathbb R}
   \trE\left[
     \bigl((\partial_{x_n}w)a_{-1}^+
                    +w(\partial_{x_n}a_{-1})^+\bigr)
                                \partial_z^2b_{3-n}\right]
                   \dd z\dd S(\xi').
\end{align*}
The normal derivatives of the right leading symbol are
\begin{align*}
 \partial_zb_{3-n}
 &=\frac{\ii c_g(dx_n)}{(1+z^2)^{m-1}}
        -\frac{2\ii(m-1)zc}{(1+z^2)^m},\\
 \partial_z^2b_{3-n}
 &=-\frac{4\ii(m-1)z c_g(dx_n)}{(1+z^2)^m}
     +\ii c\left[-\frac{2(m-1)}{(1+z^2)^m}
             +\frac{4m(m-1)z^2}{(1+z^2)^{m+1}}\right].
\end{align*}
For the part containing $\partial_{x_n}w$,
\eqref{eq:boundary-projections} gives
\[
 \trE(a_{-1}^+c_g(dx_n))=-\frac{\ii2^{n-1}}{z-\ii},
 \qquad
 \trE(a_{-1}^+c)=-\frac{2^{n-1}(1+\ii z)}{z-\ii}.
\]
The part containing $w$ is evaluated by the two traces of
$(\partial_{x_n}a_{-1})^+$ obtained above. Substitution,
followed by integration of the odd terms, yields
\begin{align*}
 \Phi_2
 ={}&2^n\nu_{n-2}(\partial_{x_n}w)
       \int_{\mathbb R}
       \frac{(m-1)[1-(2m-3)z^2]}{2(1+z^2)^{m+1}}\dd z\\
 &+2^n\nu_{n-2}w\eta
       \int_{\mathbb R}
       \frac{(m-1)[(2m-5)z^2-1]}{4(1+z^2)^{m+2}}\dd z\\
 ={}&2^n\nu_{n-2}a_m
       \left[\partial_{x_n}w
                    -\frac{3(2m-1)}{4(m+1)}w\eta\right].
\end{align*}
For example, the scalar coefficient in the second integral is
\[
 \frac{m-1}{4}
 \left[(2m-5)\frac{2m-1}{2(m+1)(m-1)}
       -\frac{(2m+1)(2m-1)}{2(m+1)(m-1)}\right]a_m
 =-\frac{3(2m-1)}{4(m+1)}a_m.
\]

\paragraph{Case (a)(III).}
Now $r=-1$, $\ell=3-n$, $k=1$, and $j=|\alpha|=0$.
The coefficient is again $-1/2$, and
\begin{align*}
 \Phi_3
 &=-\frac12\int_{|\xi'|_{g^{\partial M}}=1}\int_{\mathbb R}
  \trE\left[
    \partial_z\sigma_{-1}^+(wD^{-1})\,
    \partial_z\partial_{x_n}\sigma_{3-n}(D^{-n+3})\right]
                     \dd z\dd S(\xi')\\
 &=-\frac w2\int_{|\xi'|_{g^{\partial M}}=1}\int_{\mathbb R}
  \trE\left[(\partial_z a_{-1}^+)
                          \partial_z\partial_{x_n}b_{3-n}\right]
                     \dd z\dd S(\xi').
\end{align*}
Differentiation before restriction to the boundary gives
\begin{align*}
 \partial_{x_n}b_{3-n}
 &=-\ii\eta\left[
       \frac{\sum_{a<n}\xi_a\iot(\partial_a)}{(1+z^2)^{m-1}}
                   +\frac{(m-1)c}{(1+z^2)^m}\right],\\
 \partial_z\partial_{x_n}b_{3-n}
 &=-\ii\eta(m-1)\left[
     -\frac{2z\sum_{a<n}\xi_a\iot(\partial_a)}{(1+z^2)^m}
     +\frac{c_g(dx_n)}{(1+z^2)^m}
     -\frac{2mz c}{(1+z^2)^{m+1}}\right].
\end{align*}
Use $\partial_z a_{-1}^+
=-[c_g(\xi')+\ii c_g(dx_n)]/[2(z-\ii)^2]$ and the
two-factor traces. After the odd terms are integrated out,
\begin{align*}
 \Phi_3
 &=2^n\nu_{n-2}w\eta\int_{\mathbb R}
       \frac{(m-1)[(2m-3)z^2+1]}{4(1+z^2)^{m+2}}\dd z\\
 &=2^n\nu_{n-2}a_m\frac{(2m-1)^2}{4(m+1)}w\eta.
\end{align*}
The last equality uses
$(2m-3)+(2m+1)=2(2m-1)$ in the scalar integrals of the
preceding subsection.

\paragraph{Case (b).}
Here $r=-2$, $\ell=3-n$, and $j=k=|\alpha|=0$.
The coefficient in \eqref{eq:FGLS-new} is $-\ii$, so
\begin{align*}
 \Phi_4
 &=-\ii\int_{|\xi'|_{g^{\partial M}}=1}\int_{\mathbb R}
    \trE\left[\sigma_{-2}^+(wD^{-1})\,
                      \partial_z\sigma_{3-n}(D^{-n+3})\right]
                           \dd z\dd S(\xi')\\
 &=-\ii w\int_{|\xi'|_{g^{\partial M}}=1}\int_{\mathbb R}
                 \trE[a_{-2}^+\partial_zb_{3-n}]\dd z\dd S(\xi').
\end{align*}
By \eqref{eq:boundary-inverse},
$a_{-2}=\left.a_{-2}\right|_{\Psi=0}
+c\Psi c/(1+z^2)^2$.
For the unperturbed part,
\eqref{eq:boundary-subleading-projected} gives directly
\begin{align*}
 &-\ii\trE\left[(\left.a_{-2}\right|_{\Psi=0})^+
                                          \partial_zb_{3-n}\right]\\
 &\qquad=\ii 2^{n-2}\eta\left[
       \frac1{(z-\ii)^3(1+z^2)^{m-1}}
        -\frac{2(m-1)z}{(z-\ii)^2(1+z^2)^m}\right].
\end{align*}
Its part even in $z$ is
$2^{n-2}\eta[1+(4m-7)z^2]/(1+z^2)^{m+2}$.

For the perturbation, the Clifford relations imply
\begin{align*}
 c\,c_g(\xi')c&=(z^2-1)c_g(\xi')-2z c_g(dx_n),\\
 c\,c_g(dx_n)c&=(1-z^2)c_g(dx_n)-2z c_g(\xi').
\end{align*}
Taking the trace after multiplication by $\Psi$, and projecting,
we obtain
\begin{align*}
 \trE\left[\left(\frac{c\Psi c}{(1+z^2)^2}\right)^+
                                                c_g(\xi')\right]
 &=\frac{\trE(c_g(\xi')\Psi)+\ii\trE(c_g(dx_n)\Psi)}
                                      {2(z-\ii)^2},\\
 \trE\left[\left(\frac{c\Psi c}{(1+z^2)^2}\right)^+
                                                c_g(dx_n)\right]
 &=\frac{\ii\trE(c_g(\xi')\Psi)-\trE(c_g(dx_n)\Psi)}
                                      {2(z-\ii)^2}.
\end{align*}
The terms containing $\trE(c_g(\xi')\Psi)$ are odd in $\xi'$.
Substituting the remaining terms in $\Phi_4$ gives
\begin{align*}
 \Phi_4
 ={}&2^n\nu_{n-2}w\eta\int_{\mathbb R}
              \frac{1+(4m-7)z^2}{4(1+z^2)^{m+2}}\dd z\\
 &+\nu_{n-2}w\trE(c_g(dx_n)\Psi)\int_{\mathbb R}
              \frac{1+(2m-3)z^2}{2(1+z^2)^{m+1}}\dd z\\
 ={}&2^n\nu_{n-2}a_m\frac{3(2m-1)}{4(m+1)}w\eta
       +2\nu_{n-2}a_mw\trE(c_g(dx_n)\Psi).
\end{align*}
Both scalar coefficients follow from \eqref{eq:boundary-beta}:
in the first one $(2m+1)+(4m-7)=6(m-1)$, and in the second
$(2m-1)+(2m-3)=4(m-1)$.

\paragraph{Case (c).}
Finally, $r=-1$, $\ell=2-n$, and $j=k=|\alpha|=0$,
with coefficient $-\ii$. Thus
\begin{align*}
 \Phi_5
 &=-\ii\int_{|\xi'|_{g^{\partial M}}=1}\int_{\mathbb R}
  \trE\left[\sigma_{-1}^+(wD^{-1})\,
                       \partial_z\sigma_{2-n}(D^{-n+3})\right]
                            \dd z\dd S(\xi')\\
 &=-\ii w\int_{|\xi'|_{g^{\partial M}}=1}\int_{\mathbb R}
                 \trE[a_{-1}^+\partial_zb_{2-n}]\dd z\dd S(\xi').
\end{align*}
The symbol \eqref{eq:boundary-right-symbols} decomposes as
\[
 b_{2-n}=\left.b_{2-n}\right|_{\Psi=0}
       +\frac{\Psi}{(1+z^2)^{m-1}}
       +\frac{(m-1)\{c,\Psi\}c}{(1+z^2)^m}.
\]
For the metric part, the last two lines of
\eqref{eq:boundary-subleading-traces} imply
\begin{align*}
 &\trE\left[a_{-1}^+\partial_z
                      \left(\left.b_{2-n}\right|_{\Psi=0}\right)\right]\\
 &\quad=
 \frac{2^{n-2}\eta(m-1)}{z-\ii}\frac d{dz}
 \left[
 \frac{z[z^2-(2m-1)]+\ii[1-(2m-1)z^2]}{(1+z^2)^{m+1}}
 \right].
\end{align*}
The prefactor $(z-\ii)^{-1}$ is not differentiated, since
the derivative acts only on the right symbol.
For the perturbation part, cyclicity and $c^2=-(1+z^2)\Id_E$
give
\begin{align*}
 \trE\bigl(c_g(\xi')\{c,\Psi\}c\bigr)
    &=-2\trE(c_g(\xi')\Psi)-2z\trE(c_g(dx_n)\Psi),\\
 \trE\bigl(c_g(dx_n)\{c,\Psi\}c\bigr)
    &=-2z\trE(c_g(\xi')\Psi)-2z^2\trE(c_g(dx_n)\Psi).
\end{align*}
After integration over the tangential sphere, the traces
containing $c_g(\xi')\Psi$ vanish. Differentiating the two
remaining scalar expressions and integrating the odd terms in $z$
out therefore gives
\begin{align*}
 \Phi_5
 ={}&2^n\nu_{n-2}w\eta\int_{\mathbb R}
    \frac{(m-1)(2m-1)[(2m-1)z^2-1]}{4(1+z^2)^{m+2}}\dd z\\
 &+\nu_{n-2}w\trE(c_g(dx_n)\Psi)\int_{\mathbb R}
    \frac{(m-1)[(2m-3)z^2-1]}{(1+z^2)^{m+1}}\dd z\\
 ={}&-2^n\nu_{n-2}a_m\frac{(2m-1)^2}{4(m+1)}w\eta
        -2\nu_{n-2}a_mw\trE(c_g(dx_n)\Psi).
\end{align*}
Here the differences of the scalar numerators are
$(2m-1)-(2m+1)=-2$ and
$(2m-3)-(2m-1)=-2$, respectively.

\begin{proposition}[Weighted boundary KKW formula]
\label{prop:boundary-KKW}
Under \eqref{eq:boundary-collar}, the residue
\eqref{eq:boundary-B} is
\begin{equation}
\label{eq:boundary-B-result}
\begin{aligned}
 \mathcal B_t(f,N)
 ={}&2^n\nu_{n-1}\int_M fN
     \left(-\frac{n-2}{24}R_g-\frac{n-2}{2}v_\Psi\right)\dd\vol_g\\
 &+2^n\nu_{n-1}\frac{n-2}{2n}
       \int_{\partial M}\partial_{x_n}(fN)\dd\vol_{g^{\partial M}}.
\end{aligned}
\end{equation}
\end{proposition}

\begin{proof}
The five contributions satisfy
\begin{align*}
 \Phi_1&=0,\\
 \Phi_2+\Phi_4
    &=2^n\nu_{n-2}a_m\partial_{x_n}w
          +2\nu_{n-2}a_mw\trE(c_g(dx_n)\Psi),\\
 \Phi_3+\Phi_5
    &=-2\nu_{n-2}a_mw\trE(c_g(dx_n)\Psi).
\end{align*}
Thus the terms containing $h'(0)$ and $\Psi$ cancel, and
\[
 \Phi(wD^{-1},D^{-n+3})
   =2^n\nu_{n-2}a_m\partial_{x_n}w
   =2^n\nu_{n-1}\frac{n-2}{2n}\partial_{x_n}(fN).
\]
The interior product is $fN|D|^{2-n}$ modulo smoothing
operators. Its local density is the weighted KKW density
used in Proposition~\ref{prop:weighted-KKW}. Adding it through
\eqref{eq:boundary-split} proves the formula.
\end{proof}

\subsection{The commutator residue}

We next apply \eqref{eq:FGLS-new} to
$U=fPD^{-5}$ and $V=D^{-n+3}$.
Their orders are at most $-2$ and $3-n$. Put $r=-2-a$
and $\ell=3-n-b$, where $a,b\ge0$. The degree equation is
\[
 (-2-a)+(3-n-b)-j-k-|\alpha|-1=-n
 \quad\Longleftrightarrow\quad a+b+j+k+|\alpha|=0.
\]
Thus only $r=-2$, $\ell=3-n$, and $j=k=|\alpha|=0$
contribute, with coefficient $-\ii$. Denoting the resulting
scalar boundary coefficient by $\Phi_C$, we have
\begin{equation}
\label{eq:commutator-one-case}
 \Phi_C=-\ii\int_{|\xi'|_{g^{\partial M}}=1}\int_{\mathbb R}
 \trE\left[\pi_z^+\sigma_{-2}(fPD^{-5})\,
                         \partial_z\sigma_{3-n}(D^{-n+3})\right]
                                      \dd z\dd S(\xi').
\end{equation}
This order count is the same simplification used in the boundary
spectral-torsion calculation of
\cite[Section~4.3]{WangWang2025Torsion}. The principal symbol
of the present insertion is computed as follows.

\paragraph{The leading symbol of the left factor.}
Because $N^{-1}$ is scalar,
\begin{align*}
 \sigma_1([A,N^{-1}])
 &=\sum_j\partial_{\xi_j}(|\xi|_g^2)
                          D_{x_j}(N^{-1})\Id_E\\
 &=-2\ii g^{ij}\partial_i(N^{-1})\xi_j\Id_E.
\end{align*}
Multiplying by the principal symbols of the other factors
in $P=NA[A,N^{-1}]N$ gives
\begin{align*}
 \sigma_3(P)
 &=N|\xi|_g^2
       \bigl[-2\ii g^{ij}\partial_i(N^{-1})\xi_j\bigr]N\Id_E\\
 &=2\ii|\xi|_g^2\langle\dd N,\xi\rangle_g\Id_E.
\end{align*}
Furthermore, $D^{-5}=A^{-3}D$ modulo smoothing operators,
so $\sigma_{-5}(D^{-5})=\ii c_g(\xi)|\xi|_g^{-6}$.
The degree $-2$ symbol of their product is therefore
\begin{equation}
\label{eq:boundary-commutator-symbol}
 \sigma_{-2}(fPD^{-5})
   =-2f\langle\dd N,\xi\rangle_g c_g(\xi)|\xi|_g^{-4}.
\end{equation}
Only principal symbols occur in this identity. In particular,
neither a derivative of $f$ nor the zero-order perturbation
$\Psi$ appears in the boundary integrand.

\paragraph{Projection, trace, and integration.}
At the boundary, with $|\xi'|_{g^{\partial M}}=1$,
\[
 \langle\dd N,\xi\rangle_g
     =\sum_{a<n}(\partial_{x_a}N)\xi_a
                            +z\partial_{x_n}N.
\]
The traces involving the tangential part reduce to a scalar
function of $z$ times
$\sum_{a<n}(\partial_{x_a}N)\xi_a$:
the contractions are $\trE(c_g(\xi')^2)=-2^n$,
$\trE(c_g(dx_n)^2)=-2^n$, and
$\trE(c_g(\xi')c_g(dx_n))=0$.
That part is odd in $\xi'$ and has zero sphere integral.
For the normal part, partial fractions give
\begin{align*}
 \pi_z^+\left(\frac{-2z\,c_g(\xi'+zdx_n)}{(1+z^2)^2}\right)
 &=\frac{\ii c_g(\xi')-c_g(dx_n)}{2(z-\ii)^2}
               +\frac{\ii c_g(dx_n)}{2(z-\ii)},\\
 \pi_z^+\frac{z^2}{(1+z^2)^2}
 &=\frac1{4(z-\ii)^2}+\frac1{4\ii(z-\ii)}.
\end{align*}
The second equality follows by subtracting
$\pi_z^+(1+z^2)^{-2}$ from $\pi_z^+(1+z^2)^{-1}$;
the first also uses the projection of $z/(1+z^2)^2$.
Substitution in \eqref{eq:commutator-one-case} now gives
\begin{align*}
 \Phi_C
 ={}&-\ii f\partial_{x_n}N
   \int_{|\xi'|_{g^{\partial M}}=1}\int_{\mathbb R}
   \trE\left[
   \left(\frac{\ii c_g(\xi')-c_g(dx_n)}{2(z-\ii)^2}
                    +\frac{\ii c_g(dx_n)}{2(z-\ii)}\right)
                 \partial_z b_{3-n}\right]\dd z\dd S(\xi')\\
 ={}&2^{n-1}\nu_{n-2}f\partial_{x_n}N
   \int_{\mathbb R}\left[
     \left(\frac1{(z-\ii)^2}-\frac\ii{z-\ii}\right)
                                          \frac1{(1+z^2)^{m-1}}
    -\frac{2(m-1)z(1-\ii z)}{(z-\ii)(1+z^2)^m}
                    \right]\dd z\\
 ={}&-2^n\nu_{n-2}f\partial_{x_n}N
       \int_{\mathbb R}\frac{(2m-3)z^2}{(1+z^2)^{m+1}}\dd z\\
 ={}&-2^n\nu_{n-2}a_m\frac{2m-3}{m-1}f\partial_{x_n}N\\
 ={}&-2^n\nu_{n-1}\frac{n-3}{n}f\partial_{x_n}N.
\end{align*}
In the third line the odd part of the rational function has
zero integral. The last two lines use
\eqref{eq:boundary-beta} and \eqref{eq:boundary-am}.

\begin{proposition}[Boundary commutator correction]
\label{prop:boundary-commutator}
Under \eqref{eq:boundary-collar}, the residue
\eqref{eq:boundary-C} is
\begin{equation}
\label{eq:boundary-C-result}
\begin{aligned}
 \mathcal C_t(f,N)
 ={}&2^n\nu_{n-1}\int_M f
     \left(\frac{n+4}{n}\Div_g\grad_gN
                     -\frac4{nN}|\dd N|_g^2\right)\dd\vol_g\\
 &-2^n\nu_{n-1}\frac{n-3}{n}
       \int_{\partial M}f\partial_{x_n}N\dd\vol_{g^{\partial M}}.
\end{aligned}
\end{equation}
\end{proposition}

\begin{proof}
The interior product equals $fP|D|^{-n-2}$ modulo smoothing
operators. Its density is the pointwise formula of
Proposition~\ref{prop:third-order}, multiplied by $f$.
Adding the boundary coefficient just calculated through
\eqref{eq:boundary-split} proves the result.
\end{proof}

\subsection{The weighted formula with boundary}

The velocity operator
\[
 T_t=N^{-1}|D|^{-n-4}
       \left(\frac{n(n+2)}8(D\dot D)^2
                   +\frac{n(n-4)}8D^2\dot D^2\right)
\]
has order at most $-n$ and satisfies the transmission condition.
We use its single truncation $\pi^+(fT_t)$, whose singular
Green component is zero. Combining it with the two
factorizations above gives the following definition.

\begin{definition}[Weighted boundary functional]
\label{def:boundary-functional}
Under \eqref{eq:boundary-collar}, define
\begin{equation}
\label{eq:boundary-functional}
 \mathscr L_t^\partial(f)
 :=\Wresb[\pi^+(fT_t)]
       +\frac6{n-2}\mathcal B_t(f,N)
       +\frac{n(n-2)}{16}\mathcal C_t(f,N).
\end{equation}
\end{definition}
Equivalently, by \eqref{eq:boundary-third-identity},
\begin{align*}
 \mathscr L_t^\partial(f)
 ={}&\Wresb[\pi^+(fT_t)]
       +\left(\frac6{n-2}+\frac{n(n-2)}{16}\right)
                                       \mathcal B_t(f,N)\\
 &-\frac{n(n-2)}{16}\Wresb\left[
       \pi^+(fNAN^{-1}AN D^{-5})\circ\pi^+(D^{-n+3})\right].
\end{align*}
This is the three-term expression \eqref{eq:functional} with
the specified boundary factorizations; the equality follows
from the operator identity before truncation and linearity.

\begin{theorem}[Weighted Hamiltonian-type formula with boundary]
\label{thm:boundary-main}
Let $M^{2m}$ be compact and oriented with smooth boundary,
$m\ge2$, and let the metric family satisfy
\eqref{eq:boundary-collar}. For smooth $N>0$, a real smooth
function $f$, and the smooth self-adjoint perturbation $\Psi_t$,
the functional \eqref{eq:boundary-functional} is
\begin{equation}
\label{eq:boundary-final}
\begin{aligned}
 \mathscr L_t^\partial(f)
 ={}&2^n\nu_{n-1}\int_M
 f\left[
 \frac{(\trg\dot g)^2+2|\dot g|_g^2}{16N}
 -\frac14NR_g-3Nv_\Psi-\frac{n-2}{4N}|\dd N|_g^2
 \right]\dd\vol_g\\
 &-2^n\nu_{n-1}\frac{(n-2)(n+4)}{16}
       \int_M\langle\dd f,\dd N\rangle_g\dd\vol_g\\
 &+2^n\nu_{n-1}\int_{\partial M}
 \left[\frac3n\partial_{x_n}(fN)
       -\frac{(n-2)(2n+1)}{16}f\partial_{x_n}N\right]
                                          \dd\vol_{g^{\partial M}}.
\end{aligned}
\end{equation}
Here $v_\Psi$ is the scalar perturbation term used in
Theorem~\ref{thm:general-main}, and $\partial_{x_n}$ is
the inward unit normal along the boundary.
\end{theorem}

\begin{proof}
The velocity calculation \eqref{eq:kinetic-result} is local;
its integral over $M$ is $\Wresb[\pi^+(fT_t)]$.
Substituting \eqref{eq:boundary-B-result} and
\eqref{eq:boundary-C-result} in
\eqref{eq:boundary-functional}, the interior contribution is
\begin{align*}
 &2^n\nu_{n-1}\int_M f\left[
 \frac{(\trg\dot g)^2+2|\dot g|_g^2}{16N}
 -\frac14NR_g-3Nv_\Psi-\frac{n-2}{4N}|\dd N|_g^2
 \right]\dd\vol_g\\
 &\qquad
 +2^n\nu_{n-1}\frac{(n-2)(n+4)}{16}
       \int_M f\Div_g\grad_gN\dd\vol_g,
\end{align*}
and the singular Green contribution is
\[
 2^n\nu_{n-1}\int_{\partial M}
 \left[\frac3n\partial_{x_n}(fN)
       -\frac{(n-2)(n-3)}{16}f\partial_{x_n}N\right]
                                      \dd\vol_{g^{\partial M}}.
\]
Since the outward unit normal is $-\partial_{x_n}$,
the product rule and the divergence theorem give
\begin{align*}
 \Div_g(f\grad_gN)
     &=\langle\dd f,\dd N\rangle_g+f\Div_g\grad_gN,\\
 \int_M f\Div_g\grad_gN\dd\vol_g
     &=-\int_M\langle\dd f,\dd N\rangle_g\dd\vol_g
       -\int_{\partial M}f\partial_{x_n}N\dd\vol_{g^{\partial M}}.
\end{align*}
The coefficient of $f\partial_{x_n}N$ in the total boundary
term is consequently
\[
 -\frac{(n-2)(n-3)}{16}
 -\frac{(n-2)(n+4)}{16}
 =-\frac{(n-2)(2n+1)}{16}.
\]
This proves \eqref{eq:boundary-final}.
\end{proof}

When $\partial M=\varnothing$, the formula reduces to
Theorem~\ref{thm:general-main}. If $f=1$, its total boundary term is
\[
 2^n\nu_{n-1}\left[\frac3n-\frac{(n-2)(2n+1)}{16}\right]
       \int_{\partial M}\partial_{x_n}N\dd\vol_{g^{\partial M}}.
\]
More generally, both the singular Green contribution and the
boundary term from integration by parts vanish when
$\partial_{x_n}f=\partial_{x_n}N=0$ along $\partial M$.
The time-integrated action is
$\int_{t_0}^{t_1}\mathscr L_t^\partial(f_t)\dd t$
for a compact interval $[t_0,t_1]\subset I$.

\section*{Acknowledgements}

This work was supported by Science and Technology Development Plan Project of Jilin Province,
China: No.20260102245JC, NSFC (Grant Nos. 12301063 and 11771070) and 2024 Liaoning Provincial Natural Science Foundation Program (Ph.D. Research Start-up Project) (Grant No. 2024-BS-205).

\end{document}